\documentclass[11pt]{amsart}

\usepackage{amscd,amssymb,amsmath,graphicx,verbatim}
\usepackage{wasysym}
\usepackage{hyperref}
\usepackage[TS1,OT1,T1]{fontenc}
\usepackage{lscape} 
\usepackage{fullpage} 
\usepackage{pslatex} 
\usepackage{tikz}
\usetikzlibrary{shapes,arrows}
\usepackage[all]{xy}
\newtheorem{theorem}{Theorem}[section]
\newtheorem{lemma}[theorem]{Lemma}
\newtheorem{corollary}[theorem]{Corollary}
\newtheorem{proposition}[theorem]{Proposition}

\theoremstyle{definition}
\newtheorem{definition}[theorem]{Definition}

\newtheorem{example}[theorem]{Example}

\theoremstyle{remark}
\newtheorem{remark}[theorem]{Remark}

\newcommand{\Dcal}{\ensuremath{\mathcal{D}}}
\newcommand{\Xcal}{\ensuremath{\mathcal{X}}}
\newcommand{\Ycal}{\ensuremath{\mathcal{Y}}}
\newcommand{\Tcal}{\ensuremath{\mathcal{T}}}
\newcommand{\Gcal}{\ensuremath{\mathcal{G}}}
\newcommand{\Fcal}{\ensuremath{\mathcal{F}}}
\newcommand{\Ccal}{\ensuremath{\mathcal{C}}}

\newcommand{\Rcal}{\ensuremath{\mathcal{R}}}

\newcommand{\Wcal}{\ensuremath{\mathcal{W}}}

\newcommand{\Ucal}{\ensuremath{\mathcal{U}}}

\newcommand{\Kbb}{\mathbb{K}}

\newcommand{\ra}{\rightarrow}

\numberwithin{equation}{section}

\begin{document}
\title{Silting modules and ring epimorphisms}
\author{Lidia Angeleri H\"ugel, Frederik Marks, Jorge Vit{\'o}ria}

\address{Lidia Angeleri H\"ugel, Jorge Vit\'oria, Dipartimento di Informatica - Settore di Matematica, Universit\`a degli Studi di Verona, Strada le Grazie 15 - Ca' Vignal, I-37134 Verona, Italy} \email{lidia.angeleri@univr.it, jorge.vitoria@univr.it}
\address{Frederik Marks, Institut f\"ur Algebra und Zahlentheorie, Universit\"at Stuttgart, Pfaffenwaldring 57, 70569 Stuttgart, Germany}
\email{marks@mathematik.uni-stuttgart.de}

\maketitle

\begin{center} {Dedicated to Alberto Facchini on the occasion of his 60th birthday} \end{center}

\begin{abstract}
There are well-known constructions relating  ring epimorphisms and tilting modules. The new notion of silting module provides a wider framework for studying this interplay. To every partial silting module we associate a ring epimorphism which we describe explicitly as an idempotent quotient of the endomorphism ring of the Bongartz completion. For hereditary rings, this assignment is used to parametrise homological ring epimorphisms by  silting modules. We further show that homological ring epimorphisms of a hereditary ring form a lattice which completes the poset of noncrossing partitions in the case of finite dimensional algebras.
\end{abstract}

\bigskip

\section{Introduction}
There is a close relationship between ring epimorphisms and tilting theory which goes back to \cite{GeLe} and was further studied in \cite{AS1,AA}. In fact,  ring epimorphisms  with nice homological properties can be used to construct tilting modules. Here we have to deal with large tilting modules, since  even the ring epimorphisms $A\to B$  of a  finite dimensional algebra $A$ usually involve infinite dimensional algebras $B$.
Tilting modules arising from ring epimorphisms play an important role in classification results. For example, over a Dedekind domain, all tilting modules are of this form, and over the Kronecker algebra, this applies to all but one up to equivalence (\cite{AS1,AS2,Ma}). 

In \cite{AMV}, we developed the new notion of silting module as a common generalisation of (possibly large) tilting modules and support $\tau$-tilting modules. In this paper we show that silting theory is an appropriate context to study the phenomena above, as it provides a wider framework for the interplay with ring epimorphisms. This idea is further supported by parallel work in \cite{MS}, where it is shown that for finite dimensional algebras of finite representation type, silting modules are in bijection with universal localisations.
 
Silting modules over an arbitrary ring $A$ capture some of the main features of tilting and support $\tau$-tilting modules. As shown in \cite{AMV}, they generate torsion classes in the module category which provide left approximations, they are the  $0^{th}$-cohomologies of 2-term silting complexes, and they correspond bijectively to certain t-structures and co-t-structures in the derived category. 

Of particular relevance to this paper is the existence of a suitable notion of partial silting module. In \cite{AMV}, we proved that every partial silting module admits an analogue of the Bongartz complement and, thus, can be completed to a silting module. Here, we associate to a given partial silting module a full subcategory of $Mod(A)$ which can be understood as its perpendicular category. We show that this subcategory is bireflective, that is, its inclusion functor admits both left and right adjoints. Such subcategories are known to correspond bijectively to equivalence classes of ring epimorphisms starting in $A$ (\cite{GdP}). We can thus assign a ring epimorphism $A\to B$ to every partial silting module. 

This assignment extends results proved in the context of tilting (\cite{GeLe,CTT2}) or support $\tau$-tilting modules (\cite{J}).
In particular, it is shown in \cite{J} that the ring $B$ associated with a $\tau$-rigid (that is, a finitely generated partial silting) module $T_1$ over a finite dimensional algebra $A$ 
is a support algebra of the endomorphism ring $End_A(T)$ of the Bongartz completion $T$ of $T_1$. We prove (Theorem \ref{thm2}) that all ring epimorphisms arising from partial silting modules can be described in a similar way as idempotent quotients of  $End_A(T)$. Working with large modules implies, however, that we have to replace idempotent elements by idempotent ideals, and the proof requires a detailed analysis of the functors involved in our construction. 
We further investigate  properties of the ring epimorphism associated with a partial silting module. It turns out that even finite dimensional partial tilting modules over  finite dimensional algebras can give rise to non-injective or non-homological ring epimorphisms. 

Later, we focus on the case when $A$ is a hereditary ring. Here all ring epimorphisms arising from partial silting modules are homological.  Conversely, every homological ring epimorphism $f:A\to B$ gives rise to a silting $A$-module $T=B\oplus Coker(f)$, and the map $f$ can be regarded as a minimal left $Add(T)$-approximation of the regular module $A$. We thus restrict our attention to minimal silting modules, that is, silting modules $T$ providing a minimal $Add(T)$-approximation sequence $A\to T_0\to T_1\to 0$. Using that  $A$ is hereditary, it follows that $T_1$ is partial silting and, thus, we can assign to $T$ a well-defined ring epimorphism. We show that this assignment  establishes a bijection between minimal silting modules and homological ring epimorphisms (Theorem \ref{main silting epi}), where the minimal tilting modules correspond to injective homological ring epimorphisms.

Over a finite dimensional hereditary algebra $A$, our correspondence restricts to a bijection  between finite dimensional
support tilting modules and  homological ring epimorphisms $A\to B$ with finite dimensional $B$. These objects also correspond to  finitely generated wide  subcategories of $mod(A)$, and we recover results from  \cite{IT, Ma}.  Of course, the combinatorial interpretation  of finite dimensional support tilting modules  in terms of noncrossing partitions or clusters is lost when working in our general setting. However, a ring theoretic counterpart is provided by the poset of all homological ring epimorphisms. Since homological ring epimorphisms coincide with universal localisations (\cite{KSt}), this poset turns out to be a lattice  (Corollary \ref{lattice}). Notice that  if we restrict to ring epimorphisms with finite dimensional target, we obtain the poset of exceptional antichains from \cite{R} (see also \cite{HK}) which is not a lattice in general.

Finally, we show that our classification of minimal tilting modules over hereditary rings    fits in a number of  classification results over  further  rings that 
reveal a deep connection between tilting theory and {localisation}. 

In a forthcoming paper, we will explore the connections between silting theory and categorical localisation of  module categories and derived categories. An important role will be played by the explicit description of the ring epimorphism associated with a partial silting module  mentioned above.

The structure of the paper is as follows. Section 2 recalls some facts on silting modules and ring epimorphisms. In Section 3, we construct the bireflective subcategory associated with a partial silting module and we describe explicitly the associated ring epimorphism. In Section 4, we investigate the homological properties of these ring epimorphisms. Finally, Section 5 discusses silting modules over hereditary rings and the  classification of homological ring epimorphisms.  

\smallskip

{\bf Acknowledgement.}  
The first named author is partially supported by  Fondazione Cariparo, Progetto di Eccellenza ASATA.
The third named author is supported by a Marie Curie Intra-European
Fellowship within the 7th European Community Framework Programme
(PIEF-GA-2012-327376).

\section{Preliminaries}

Throughout, $A$ is a unitary ring and by an $A$-module we mean a right $A$-module, unless otherwise stated.
The category of all $A$-modules will be denoted by $Mod(A)$ and its full subcategory of projective $A$-modules by $Proj(A)$. The category $mod(A)$ denotes the category of finitely presented modules.
In some contexts, we will consider finite dimensional algebras over an algebraically closed field $\Kbb$, usually denoted by $\Lambda$.

For a given $A$-module $M$, we denote by $M^\circ$ the subcategory of $Mod(A)$ consisting of the objects $N$ such that $Hom_A(M,N)=0$, and by $M^{\perp_1}$ the subcategory of $Mod(A)$ consisting of the objects $N$ such that $Ext_A^1(M,N)=0$. We set $M^\perp$ to be $M^\circ\cap M^{\perp_1}$.
Further, $Add(M)$ denotes the additive closure of $M$ consisting of all modules isomorphic to a direct summand of an arbitrary direct sum of copies of $M$, while $Gen(M)$ is the subcategory of $M$-generated modules (that is, all epimorphic images of modules in $Add(M)$).

\subsection{Silting modules}
Recall from \cite{CT} that an $A$-module $T$ is called \textbf{partial tilting}, if $T^{\perp_1}$ is a torsion class containing $T$. Moreover, $T$ is \textbf{tilting}, if $Gen(T)=T^{\perp_1}$. In order to introduce the notion of (partial) silting module, consider, for a morphism $\sigma$ in $Proj(A)$, the class of $A$-modules
$$\Dcal_\sigma:=\{X\in Mod(A)|Hom_A(\sigma,X)\ \text{is surjective}\}.$$

\begin{definition}\cite[Definition 3.10]{AMV}\label{def p silting}
We say that an $A$-module $T$ is
\begin{itemize}
\item  \textbf{partial silting} if there is a projective presentation $\sigma$ of $T$ such that 
\begin{enumerate}
\item[(S1)] $\Dcal_\sigma$ is a torsion class.
\item[(S2)] $T$ lies in $\Dcal_\sigma$.
\end{enumerate}
\item \textbf{silting} if there is a projective presentation $\sigma$ of $T$ such that $Gen(T)=\Dcal_\sigma$.
\end{itemize}
We will then say that $T$ is (partial) silting \textbf{with respect to} $\sigma$.
\end{definition}

It follows easily from the definition that (partial) tilting modules are (partial) silting. Recall that an object $M$ in a class $\Ccal$ of $A$-modules is said to be {\bf Ext-projective in $\Ccal$}, if $\Ccal$ is contained in $M^{\perp_1}$. The following facts on silting modules will be useful later.

\begin{proposition}\cite[Lemma 3.4 and 3.7, Proposition 3.5 and 3.13]{AMV}\label{tilt silt}
Let $T$ be a silting $A$-module. The following statements hold.
\begin{enumerate}
\item $T$ is tilting over $A/Ann(T)$.
\item $Add(T)$ is the class of Ext-projective modules in $Gen(T)$.
\item There is an exact sequence
$$\xymatrix{A\ar[r]^\phi& T_0\ar[r]&T_1\ar[r]&0}$$
such that $T_0$ and $T_1$ lie in $Add(T)$ and $\phi$ is a left $Gen(T)$-approximation. 
\end{enumerate}
\end{proposition}

In particular, from (2) above it follows that two silting modules have the same additive closure if and only if they generate the same torsion class, in which case we say that they are \textbf{equivalent}. Moreover, from (3) it is easy to see that every $A$-module admits a left $Gen(T)$-approximation.

\subsection{Ring epimorphisms}
Recall that a \textbf{ring epimorphism} is an epimorphism in the category of rings with unit. Two ring epimorphisms $f:A\rightarrow B$ and $g:A\rightarrow C$ are said to be \textbf{equivalent} if there is a ring isomorphism $h: B\rightarrow C$ such that $g=h\circ f$. We then say that $B$ and $C$ lie in the same \textbf{epiclass} of $A$.
Epiclasses of a ring $A$ can be classified by suitable subcategories of $Mod(A)$. For a ring epimorphism $f:A\ra B$ we denote by $\mathcal{X}_B$ the essential image of the associated restriction functor $f_*$.

\begin{theorem}\cite[Theorem 1.2]{GdP}\cite{GeLe}\label{bireflective}
There is a bijection between:
\begin{enumerate}
\item epiclasses of ring epimorphisms $A\ra B$;
\item bireflective subcategories $\Xcal_B$ of $Mod(A)$, i.e., full subcategories of $Mod(A)$ closed under products, coproducts, kernels and cokernels.
\end{enumerate}
\end{theorem}

Bireflective subcategories $\Xcal$ are precisely those for which the inclusion functor $\Xcal\rightarrow Mod(A)$ admits both a left and a right adjoint. As a consequence, there are left $\Xcal$-approximations $\psi_M:M\rightarrow X_M$, for all $M$ in $Mod(A)$, such that $Hom_A(\psi_M,X)$ is an isomorphism for all $X$ in $\Xcal$. These approximations are called $\Xcal$-\textbf{reflections} and they are left-minimal, i.e., any endomorphism $\theta$ of $X_M$ with $\theta\circ\psi_M=\psi_M$ is an isomorphism. In fact, if $f:A\rightarrow B$ is the ring epimorphism associated with $\Xcal$, then $\psi_M$ is the natural map $M\rightarrow M\otimes_AB$. Dually, there are right $\Xcal$-approximations with analogous properties called $\Xcal$-\textbf{coreflections}.

The following two results will be frequently used throughout.

\begin{proposition}\cite[Theorem 4.8]{Sch}\label{Tor Ext}
Let $A\ra B$ be a ring epimorphism. Then the following are equivalent.
\begin{enumerate}
\item $Tor_1^A(B,B)=0$;
\item $Ext_A^1(M,N)\cong Ext_B^1(M,N)$ for all $B$-modules $M$ and $N$.
\end{enumerate}
\end{proposition}

\begin{proposition}\label{surjective ring epi}
Let $\pi:A\ra\bar{A}$ be a surjective ring epimorphism with kernel $I$. The following holds.
\begin{enumerate}
\item The subcategory $\Xcal_{\bar{A}}$ is closed under quotients and subobjects in $Mod(A)$.
\item $I$ is idempotent if and only if $\Xcal_{\bar{A}}$ is closed under extensions in $Mod(A)$. In this case, we have $\Xcal_{\bar{A}}=I^\circ$. 
\end{enumerate}
\end{proposition} 

\begin{proof}
(1)  This follows easily from the fact that $\Xcal_{\bar{A}}=\{X\in Mod(A)\mid XI=0\}$.

(2) First observe that $Tor_1^A(\bar{A},\bar{A})=Tor_1^A(A/I,A/I)\cong I/I^2$. The first part of the statement follows then from Proposition \ref{Tor Ext}. Assume now that $I=I^2$. By applying the functor $Hom_A(-,X)$, for an $A$-module $X$, to the short exact sequence induced by $\pi$, we get the exact sequence
$$\xymatrix{0\ar[r] & Hom_A(\bar{A},X)\ar[r]^{\pi_X} & Hom_A(A,X)\ar[r] & Hom_A(I,X)\ar[r] & Ext_A^1(\bar{A},X)\ar[r] & 0}$$
If $X\in\Xcal_{\bar{A}}$, then $\pi_X$ is an isomorphism and, since $Ext_A^1(\bar{A},X)=0$ by Proposition \ref{Tor Ext}, it follows that $X\in I^\circ$. Conversely, if $Hom_A(I,X)=0$, then $\pi_X$ is an isomorphism turning $X$ into an $\bar{A}$-module.
\end{proof}

A ring epimorphism $f:A\ra B$ is said to be \textbf{homological} if for all $i>0$ we have $Tor_i^A(B,B)=0$ or, equivalently, if $Ext^i_B(M,N)\cong Ext^i_A(M,N)$ for all $M$ and $N$ in $Mod(B)$ (see \cite[Theorem 4.4]{GeLe}). Certain homological ring epimorphisms of $A$ induce tilting modules.

\begin{definition}\label{arising} A tilting $A$-module $T$ is said to \textbf{arise from a ring epimorphism}, if there is an injective ring epimorphism $A\ra B$ such that $B\oplus B/A$ is a tilting $A$-module equivalent to $T$.
\end{definition}

In this definition, the ring epimorphism is unique up to equivalence. Moreover, the canonical sequence
\[\xymatrix{0\ar[r]& A\ar[r]^f& B\ar[r]&B/A\ar[r]&0.}\] 
is an approximation sequence as in Proposition \ref{tilt silt}(3).

The following theorem relates ring epimorphisms and tilting modules.

\begin{theorem}\cite[Theorem 3.5, Theorem 3.10]{AS1}\label{ring epis and tilting}
\begin{enumerate}
\item Let $A\ra B$ be an injective homological ring epimorphism such that the $A$-module $B$ has projective dimension at most one. Then $B\oplus B/A$ is a tilting $A$-module and $\Xcal_B$ equals $(B/A)^\perp$.
\item Let $T$ be a tilting $A$-module. Then $T$ arises from a ring epimorphism if and only if there is an $Add(T)$-approximation sequence
\[\xymatrix{0\ar[r]& A\ar[r]&T_0\ar[r]&T_1\ar[r]&0}\]
such that $Hom_A(T_1,T_0)=0$.
\end{enumerate}
\end{theorem}

\section{Ring epimorphisms arising from partial silting modules}
We start by generalising some ideas from \cite{CTT2} on partial tilting modules. We fix a partial silting $A$-module $T_1$ with associated torsion class $\Dcal_\sigma$ given by a projective presentation $\sigma$ of $T_1$. Since $Gen(T_1)$ is a torsion class by \cite[Lemma 2.3]{AMV}, there are two torsion pairs associated with $T_1$:
$$(\Dcal,\Rcal):=(\Dcal_\sigma, \Dcal_\sigma^{\circ})\,\,\,\text{ and }\,\,\,
(\Tcal,\Fcal):=(Gen(T_1), T_1^{\circ}).$$

We are interested in the full subcategory $\Ycal:=\Dcal\cap\Fcal$ of $Mod(A)$. Note that, by definition,
$$\Fcal=\{X\in Mod(A)|Hom_A(\sigma,X)\text{ is injective}\}$$ 
and, therefore,
$$\Ycal=\{X\in Mod(A)|Hom_A(\sigma,X)\text{ is bijective}\}.$$
We will show that $\Ycal$ is a bireflective subcategory of $Mod(A)$ and, thus, we can associate a ring epimorphism $A\ra B$ such that $\Ycal=\Xcal_B$.

\begin{remark}\label{rem hom tensor} 
Given a ring epimorphism $f:A\ra B$ such that $\Xcal_B=\Ycal$, it follows that $\sigma\otimes_AB$ is an isomorphism (compare to \cite[Theorem 5.2]{Sch2} and \cite[Proposition 3.3(1)]{Ma}). Indeed, if $Hom_A(\sigma,X)$ is an isomorphism for all $X$ in $\Ycal$, using the adjunction $(-\otimes_AB,f_*)$, then so is $Hom_B(\sigma\otimes_AB,X)$. Hence, $Hom_B(T_1\otimes_AB,X)=0$ for all $X$ in $\Ycal$, showing that $T_1\otimes_AB=0$ and that $\sigma\otimes_AB$ is surjective. Now, if we write $\sigma:P\rightarrow Q$, then $Hom_B(\sigma\otimes_AB,P\otimes_AB)$ is an isomorphism. Consequently, the identity map on $P\otimes_AB$ factors through $\sigma\otimes_AB$, proving that $\sigma\otimes_AB$ is also injective, as wanted. Note that, in case $\sigma$ is a map between finitely generated projective modules, $B$ is the universal localisation of $A$ at $\{\sigma\}$ (see Theorem \ref{def:universallocalisation}).
\end{remark}

The following arguments mimic the approach taken in \cite[Proposition 1.4]{CTT2}.

\begin{lemma}\label{closure}
Consider an $A$-module $M$ in $\Ycal$ together with a short exact sequence in $Mod(A)$
$$\xymatrix{0\ar[r] & L\ar[r] & M\ar[r] & N\ar[r] & 0.}$$
Then we have the following equivalent conditions
$$L\in\Ycal\Leftrightarrow L\in\Dcal\Leftrightarrow N\in\Fcal\Leftrightarrow N\in\Ycal.$$
\end{lemma}

\begin{proof}
Since $M$ belongs to $\Ycal=\Dcal\cap\Fcal$, we know that $L$ lies in $\Fcal$ and $N$ lies in $\Dcal$. This proves the two outer equivalences. For the remaining one consider the following commutative diagram induced by $\sigma:P\ra Q$
$$\xymatrix{0\ar[r] & Hom_A(Q,L)\ar[r]\ar[d]^{Hom_A(\sigma,L)} & Homa_A(Q,M)\ar[r]\ar[d]^{Hom_A(\sigma,M)}_{\cong} & Hom_A(Q,N)\ar[r]\ar[d]^{Hom_A(\sigma,N)} & 0\\ 0\ar[r] & Hom_A(P,L)\ar[r] & Hom_A(P,M)\ar[r] & Hom_A(P,N)\ar[r] & 0.}$$
By the Snake Lemma, $Hom_A(\sigma,L)$ is surjective if and only if $Hom_A(\sigma,N)$ is injective.
\end{proof}

We can construct left $\Ycal$-approximations for any $A$-module $X$. Take a $\Dcal$-approximation sequence 
$$\xymatrix{X\ar[r]^{\phi_X} & M_X\ar[r] & T_1^{(I)}\ar[r] & 0,}$$
which is constructed as in \cite[Theorem 3.15]{AMV}. Recall that $T:=T_1\oplus M_A$ is a silting module with $Gen(T)=\Dcal$.
Now, consider the composition 
$$\psi_X:\xymatrix{X\ar[r]^{\phi_X} & M_X\ar@{>>}[r]^{q_X\ \ \ \ \ \ \ \ \ \ \ \ \ \ \ } & M_X/\tau_{T_1}(M_X)=:\overline{M}_X}$$
where $\tau_{T_1}$ denotes the trace of $T_1$, which is the torsion-radical with respect to $\Tcal$. Note that it is clear by construction that $\overline{M}_X$ lies in $\Ycal$ for any $A$-module $X$ and, moreover, since $\psi_X$ is the composition of a left $\Dcal$-approximation and a left $\Fcal$-approximation, it is a left $\Ycal$-approximation.

\begin{proposition}\label{prop bireflective}\label{reflections}
The full subcategory $\Ycal$ of $Mod(A)$ is bireflective and extension-closed.  Moreover, the $\Ycal$-reflection of an $A$-module $X$ is given by $\psi_X$.
\end{proposition}
\begin{proof}
We have to show that $\Ycal=\{X\in Mod(A)|Hom_A(\sigma,X)\text{ is bijective}\}$ is closed under products, coproducts, kernels, cokernels and extensions. Clearly, $\Ycal$ is closed under extensions and coproducts, since so are $\Dcal$ and $\Fcal$. Note that $\Ycal$ is also closed under products, since $Hom_A(\sigma,\prod X_i)=\prod Hom_A(\sigma,X_i)$ for $X_i$ in $Mod(A)$ and products are exact. Finally, take a map $\omega:M\ra N$ in $\Ycal$. Clearly, $Im(\omega)$ belongs to $\Ycal$, since it is a quotient of $M$ (thus, in $\Dcal$) and a submodule of $N$ (thus, in $\Fcal$). Now the claim follows by Lemma \ref{closure}.

It remains to show that $Hom_A(\psi_X,Y)$ is an isomorphism for all $Y$ in $\Ycal$. It is clearly a surjection, since $\psi_X$ is a left $\Ycal$-approximation. 
Moreover, we have $Ker(Hom_A(\psi_X,Y))=Hom_A(Coker(\psi_X),Y)$. Since $Y$ lies in $\Fcal$, it suffices to show that $Coker(\psi_X)$ is in $\Tcal=Gen(T_1)$. The following commutative diagram with surjective vertical maps finishes the proof
$$\xymatrix{X\ar[r]^{\phi_X}\ar@{=}[d] & M_X\ar[r]\ar[d] & T_1^{(I)}\ar[r]\ar[d] & 0 \\ X\ar[r]^{\psi_X} & \overline{M}_X\ar[r] & Coker(\psi_X)\ar[r] & 0.}\vspace{-0.4cm}$$
\end{proof}

Hence, it follows from Theorem \ref{bireflective} and Proposition \ref{Tor Ext} that we can associate to every partial silting module a ring epimorphism $A\ra B$ with $Tor_1^A(B,B)=0$. Moreover, we can choose $B$ to be $End_A(M_A/\tau_{T_1}(M_A))$ (see, for example, \cite{AS1} for details).
These ring epimorphisms are closely related to the ring epimorphisms built in \cite{CTT}, as the following lemma shows.

\begin{lemma}\label{CTT compare}
Let $T_1$ be a partial silting module and $T=T_1\oplus M_A$ its Bongartz completion to a silting module as above. Let $f:A\rightarrow B$ be the ring epimorphism associated with $T_1$ and $\pi:A\rightarrow \bar{A}:=A/Ann(T)$ the canonical projection. Then there is a ring epimorphism $g:\bar{A}\rightarrow B$ satisfying $f=g\circ \pi$ and such that $$Im(g_*)=Ker(Hom_{\bar{A}}(T_1,-))\cap Ker(Ext^1_{\bar{A}}(T_1,-)).$$
\end{lemma}
\begin{proof}
Since $Gen(T)$ is contained in $Im(\pi_*)$, we can naturally identify $\Ycal$ with a bireflective subcategory of $Mod(\bar{A})$. Let $g:\bar{A}\rightarrow B$ be the associated ring epimorphism. It clearly follows that $f=g\circ \pi$. Now, the left $Gen(T)$-approximation $\phi_A:A\rightarrow M_A$ induces a short exact sequence in $Mod(\bar{A})$ of the form
$$\xymatrix{0\ar[r] & \bar{A}\ar[r]^{\phi_A} & M_A\ar[r] & T_1^{(I)}\ar[r] & 0.}$$
Since, by Proposition \ref{tilt silt}(1), $T$ is a tilting $\bar{A}$-module, it follows that $Gen(T)=Ker(Ext^1_{\bar{A}}(T_1,-))$ in $Mod(\bar{A})$. By the definition of $\Ycal$, it then follows that $Im(g_*)=Ker(Hom_{\bar{A}}(T_1,-))\cap Ker(Ext^1_{\bar{A}}(T_1,-)),$ as wanted.
\end{proof}

The following theorem describes explicitly the ring epimorphism arising from a partial silting module.

\begin{theorem}\label{thm2}
Let $T_1$ be a partial silting module, $T=T_1\oplus M_A$ its Bongartz completion to a silting module as above and $f:A\rightarrow B$ the ring epimorphism associated with $T_1$. Then there is an isomorphism of rings between $B$ and $End_A(T)/I$, where $I$ is the two-sided ideal given by the endomorphisms of $T$ factoring through an object in $Add(T_1)$. Furthermore, the ideal $I$ is idempotent.
\end{theorem}

\begin{proof}
For simplicity, throughout this proof we denote $M_A$ by $M$.  According to the discussion above, we fix the ring $B$ to be $End_A(\overline{M})$, where $\overline{M}=M/\tau_{T_1}(M)$. We prove our theorem in several steps. First, we define a surjective ring homomorphism $p:End_A(T)\ra End_A(\overline{M})$, thus yielding an isomorphism of rings $\bar{p}:End_A(T)/Ker(p)\rightarrow End(\overline{M})$. In a second step, we will show that $Ker(p)$ coincides with the ideal $I$ of endomorphisms of $T$ factoring through an object in $Add(T_1)$. Steps 3 and 4 show that $I$ is idempotent.

\textbf{Step 1: A surjective ring homomorphism.} Consider the short exact sequence
$$\xymatrix{0\ar[r] & \tau_{T_1}(M)\ar[r]^{\,\,\,\,\,i} & M\ar[r]^q & \overline{M}\ar[r] & 0.}$$ Note that the trace of $T_1$ in $M$ will be preserved by any endomorphism of $M$. Consequently, for all $a$ in $End_A(M)$ there is a unique endomorphism $\bar{a}$ in $End_A(\overline{M})$ such that $\bar{a}\circ q=q\circ a$. Endomorphisms of $M$ lie in \textit{a corner} of the endomorphism ring of $T$, which can be written as a $2\times 2$-matrix as follows
$$End_A(T)=\left( \begin{array}{cc}
End_A(M) & Hom_A(T_1,M) \\
Hom_A(M,T_1) & End_A(T_1) \end{array} \right).$$
We define the following map
$$\xymatrix{p:End_A(T)\ar[r]& End(\overline{M})=B}, \ \ \ \ \ \left( \begin{array}{cc} a & b \\ c & d \end{array} \right)\mapsto \bar{a}.$$
It is straightforward to check that $p$ is a ring homomorphism. To show that it is surjective, it is enough to see that endomorphisms of $\overline{M}$ lift to endomorphisms of $M$.  Indeed, by applying the functor $Hom_A(M,-)$ to the short exact sequence above, we get the exact sequence
$$\xymatrix{End_A(M)\ar[r] & Hom_A(M,\overline{M})\ar[r] & Ext_A^1(M,\tau_{T_1}(M)).}$$
Since $M\in Add(T)$ and $\tau_{T_1}(M)\in Gen(T)$, by Proposition \ref{tilt silt}(2), we know that $Ext_A^1(M,\tau_{T_1}(M))=0$. Hence, for any $\delta$ in $End_A(\overline{M})$ there is a morphism $a$ in $End_A(M)$ such that $\delta\circ q=q\circ a$, showing that $p$ is surjective.

\textbf{Step 2: Computing the kernel.} We now show that an endomorphism $\gamma$ of $T$ belongs to the kernel of $p$ if and only if it factors through a module in $Add(T_1)$. Write 
$$\gamma=\left( \begin{array}{cc} a & b \\ c & d \end{array} \right)$$ 
and suppose first that $\gamma$ factors through $Add(T_1)$. In particular, the image of $\gamma$ lies in the trace of $T_1$ in $T$. This shows that $Im(a)\subseteq \tau_{T_1}(M)$ and, thus, $\bar{a}\circ q = q\circ a=0$. Since $q$ is surjective, $p(\gamma)=\bar{a}=0$, as wanted. Conversely, if $\gamma$ lies in the kernel of $p$, we get $q\circ a=0$, meaning that the image of $a$ lies in the trace of $T_1$. Hence, there is a map $\omega:M\ra\tau_{T_1}(M)$ making the following diagram commute
$$\xymatrix{& & M\ar[d]^{a}\ar[ld]_\omega & & \\ 0\ar[r] & \tau_{T_1}(M)\ar[r]^i & M\ar[r]^q & \overline{M}\ar[r] & 0.}$$
Let $J$ be the set $Hom_A(T_1,\tau_{T_1}(M))$ and consider the short exact sequence induced by the universal map $\mu$
$$\xymatrix{0\ar[r] & Ker(\mu)\ar[r] & T_1^{(J)}\ar[r]^{\!\!\!\!\!\!\!\mu} & \tau_{T_1}(M)\ar[r] & 0.}$$ 
Using the presentation $\sigma:P\ra Q$ of $T_1$ defining $\Dcal_\sigma=Gen(T)$, we get the following commutative diagram
\begin{equation}\nonumber
\xymatrix{& 0\ar[d] & 0\ar[d] & 0\ar[d] & \\ 0\ar[r] & Hom_A(T_1,Ker(\mu))\ar[r]\ar[d] & Hom_A(T_1,T_1^{(J)})\ar[r]^{\!\!\!\!\!\kappa}\ar[d] & Hom_A(T_1,\tau_{T_1}(M))\ar[d]\ar[r] & 0 \\ 0\ar[r] & Hom_A(Q,Ker(\mu))\ar[r]\ar[d]^{\eta} & Hom_A(Q,T_1^{(J)})\ar[r]\ar[d] & Hom_A(Q,\tau_{T_1}(M))\ar[r]\ar[d] & 0\\ 0\ar[r] & Hom_A(P,Ker(\mu))\ar[r] & Hom_A(P,T_1^{(J)})\ar[r]\ar[d] & Hom_A(P,\tau_{T_1}(M))\ar[r]\ar[d] & 0 \\ & & 0 & 0 & &}
\end{equation}
where $\kappa=Hom_A(T_1,\mu)$ is surjective by the definition of $\mu$. By the Snake Lemma, it follows that $\eta$ is surjective and $Ker(\mu)$ lies in $\Dcal_\sigma$. Consequently, since $M$ is in $Add(T)$, we get $Ext_A^1(M,Ker(\mu))=0$, using Proposition \ref{tilt silt}(2). Hence, there is a map $\nu:M\ra T_1^{(J)}$ such that $\omega=\mu\circ\nu$. We conclude that $a=i\circ\omega=i\circ\mu\circ\nu$ factors through $Add(T_1)$. Now, $\gamma$ can be decomposed as the sum of four endomorphisms of $T$ induced naturally from the maps $a$, $b$, $c$ and $d$. It is clear that the endomorphisms induced by $b$, $c$ and $d$ factor through $Add(T_1)$ and by the arguments above so does the endomorphism induced by $a$. Since the endomorphisms factoring through $Add(T_1)$ form an ideal of $End_A(T)$, it follows that $\gamma$ lies in that ideal, as wanted.

\textbf{Step 3: A commutative diagram of functors.} 
The steps above show that the canonical projection $End_A(T)\rightarrow End_A(T)/Ker(p)$ lies in the same epiclass as $p:End_A(T)\rightarrow B$. Hence, in order to show that the ideal $I=Ker(p)$ is idempotent, it is enough to check that $Im(p_*)$ is extension-closed in $Mod(End_A(T))$ (see Proposition \ref{surjective ring epi}(2)).
In this step, we show that the restriction functors $p_*$ and $f_*$ satisfy the relation $p_*\cong Hom_A(T,f_*(-))$, i.e., there is a commutative diagram of functors
$$\xymatrix{& Mod(A)\ar[dd]^{Hom_A(T,-)}\\ Mod(B)\ar[ru]^{f_*}\ar[rd]^{p_*}& \\ & Mod(End_A(T)). }$$
Recall that $p_*$ and $f_*$ can be rewritten as $Hom_B(B,-)$, where $B$ is regarded, respectively, as a left $End_A(T)$-module via $p$ and as a left $A$-module via $f$. Hence, using the adjunction
$$Hom_A(T,f_*(-))\cong Hom_B(T\otimes_AB,-)$$
it is enough to check that $B$ and $T\otimes_AB$ are isomorphic as left $End_A(T)$-modules. 
Since $T=M\oplus T_1$ and, by construction $Hom_A(T_1,\Xcal_B)=0$, we have $T\otimes_AB = M\otimes_AB$. But by Proposition \ref{prop bireflective} we have $M\otimes_AB\cong \overline{M}$, so there is an isomorphism $s:T\otimes_AB\rightarrow B$ of right $A$-modules given by $s((m,t_1)\otimes x)=q(m)x$, for $m\in M$, $t_1\in T_1$ and $x\in B$. We check that $s$ is also a map of left $End_A(T)$-modules. On one hand, we have that 
$$s(\left(\begin{array}{cc} a & b \\ c & d \end{array} \right)(m,t_1)\otimes x)=s((a(m)+b(t_1),c(m)+d(t_1))\otimes x)=q(a(m))x.$$
On the other hand, using the left action of $End_A(T)$ on $B$ via $p$ we have, as wanted,
$$\left(\begin{array}{cc} a & b \\ c & d \end{array} \right)q(m)x=\bar{a}q(m)x=q(a(m))x.$$

\textbf{Step 4: Reducing to the tilting case.}
Consider now the ring $\bar{A}=A/Ann(T)$ and the canonical projection $\pi:A\rightarrow \bar{A}$. By Lemma \ref{CTT compare} and Step 3, we get that $p_*\cong Hom_{\bar{A}}(T,\bar{f}_*(-))$, where $\bar{f}:\bar{A}\rightarrow B$ is the ring epimorphism given by the factorisation $f=\bar{f}\circ \pi$.  Since $T$ is an $\bar{A}$-tilting module (see Proposition \ref{tilt silt}(1)), we may assume without loss of generality that $T$ is a tilting $A$-module. 
Then the following useful facts hold true for  $T$.
\begin{enumerate}
\item \cite[Proposition 3.4]{B}  The functor $Tor_1^{End_A(T)}(Hom_A(T,-),T)$ is identically zero.
\item \cite[Corollary 2.18]{CT} The endofunctor $Hom_A(T,-)\otimes_{End_A(T)}T$ of $Mod(A)$ acts as the identity functor on $Gen(T)$.
\end{enumerate}

We  now show that $Im(p_*)$ is extension-closed in $Mod(End_A(T))$.
Let $X$ and $Z$ lie in $Mod(B)$ and consider a short exact sequence in $Mod(End_A(T))$ of the form
$$\epsilon: \ \ \ \ \ 0\rightarrow p_*(X)\rightarrow Y\rightarrow p_*(Z)\rightarrow 0.$$ 
Since $p_*\cong Hom_A(T,f_*(-))$, it follows from (1) above that we get a short exact sequence in $Mod(A)$
$$\epsilon\otimes_{End_A(T)}T:\ \ \ \ \ 0\rightarrow Hom_A(T,f_*(X))\otimes_{End_A(T)}T\rightarrow Y\otimes_{End_A(T)}T\rightarrow Hom_A(T,f_*(Z))\otimes_{End_A(T)}T\rightarrow 0.$$
By (2) above, the two outer terms of this sequence are isomorphic to $f_*(X)$ and $f_*(Y)$. Since $Im(f_*)$ is extension-closed in $Mod(A)$, also $Y\otimes_{End_A(T)}T$ is in $Im(f_*)$. 
By Proposition \ref{tilt silt}(2), the functor $Hom_A(T,-)$ is exact for sequences in $Gen(T)$, so $Hom_A(T,\epsilon\otimes_{End_A(T)}T)$ is a short exact sequence in $Mod(End_A(T))$. Now, the unit of the adjunction $(-\otimes_{End_A(T)}T,Hom_A(T,-))$ yields a map of short exact sequences  $$\epsilon\rightarrow  Hom_A(T,\epsilon\otimes_{End_A(T)}T)$$ which, again by fact (2), induces an isomorphism $$Y\cong Hom_A(T,Y\otimes_{End_A(T)}T).$$ Hence, $Y$ lies in $Im(Hom_A(T,f_*(-)))=Im(p_*).$
\end{proof}

\begin{remark}\label{generate small}
If every element of the ideal $I$ is an endomorphism of $T$ factoring through $add(T_1)$, then $I$ is generated by the idempotent element of $End_A(T)$ corresponding to the summand $T_1$ of $T$. This holds, for example, if $T$ is a finitely generated silting module over a finite dimensional algebra. This case has been explored in \cite{J}. 
\end{remark}

We finish this section with a reduction result for torsion classes that, in the context of partial tilting modules, was first proved in \cite[Theorem 4.4]{CTT}. A similar result was shown for $\tau$-rigid modules and their completion to support $\tau$-tilting modules over finite dimensional algebras in \cite[Theorems 3.12 and 3.13]{J}.
Given two full subcategories $\Xcal$ and $\Ycal$ of $Mod(A)$, we denote by $\Xcal\star\Ycal$ the full subcategory containing the $A$-modules $M$ such that there are $X$ in $\Xcal$ and $Y$ in $\Ycal$ and a short exact sequence of the form
$$\xymatrix{0\ar[r] & X\ar[r] & M\ar[r] & Y\ar[r] & 0.}$$
Recall that an $A$-module $M$ is called \textbf{finendo} if it is a finitely generated module over its endomorphism ring. Equivalently, by \cite[Proposition 1.2]{ATT}, every $A$-module has a left $Gen(M)$-approximation. A class $\Tcal$ of $A$-modules such that every module has a left $\Tcal$-approximation is called \textbf{preenveloping}.

\begin{theorem}
Let $T_1$ be a partial silting $A$-module with associated ring epimorphism $f:A\ra B$ and let $T$ be the completion of $T_1$ to a silting $A$-module as above. Then the following holds.
\begin{enumerate}
\item There is a bijection between
$$\left\{\begin{array}{c}\text{torsion classes\,\,} \Tcal\,in\,\, Mod(A) \\ with\,\, Gen(T_1)\subseteq\Tcal\subseteq Gen(T)\end{array}\right\}\longleftrightarrow\left\{\begin{array}{c}\text{torsion classes\,\,}\\\ in\,\, Mod(B) \end{array}\right\}$$ 
where a torsion class $\Tcal$ in $Mod(A)$ is mapped to $\Tcal\otimes_AB$. Conversely, a torsion class $\Gcal$ in $Mod(B)$ is mapped to $Gen(T_1)\star f_*(\Gcal)$.\\
\item If the $A$-module $T_1$ is finendo, then the bijection in (1) restricts to a bijection between
$$\left\{\begin{array}{c}\text{preenveloping torsion classes\,\,} \Tcal\\ in\,\, Mod(A)\,\, with\,\, Gen(T_1)\subseteq\Tcal\subseteq Gen(T)\end{array}\right\}\longleftrightarrow\left\{\begin{array}{c}\text{preenveloping torsion}\\ classes\,\, in\,\, Mod(B) \end{array}\right\}.$$ 
\end{enumerate}
\end{theorem}

\begin{proof}
Let $\Tcal$ be a torsion class in $Mod(A)$ fulfilling $Gen(T_1)\subseteq\Tcal\subseteq Gen(T)$. Since $\Xcal_B=Gen(T)\cap T_1^{\circ}$, we have $\Tcal\otimes_AB=(\Tcal\cap T_1^\circ)\otimes_AB$, which can be identified with $\Tcal\cap T_1^\circ$ in $Mod(A)$. Now the first statement follows by the same arguments used in the proof of \cite[Theorem 3.12]{J}. 

For statement (2), take a preenveloping torsion class $\Tcal\subseteq Mod(A)$ with $Gen(T_1)\subseteq\Tcal\subseteq Gen(T)$. Since  $T_1^\circ$ is a torsion-free class and $\Tcal$ is a preenveloping torsion class in $Mod(A)$, every $A$-module admits a left $\Tcal\cap T_1^\circ$-approximation, as argued before Proposition \ref{prop bireflective}. In particular, the torsion class $\Tcal\otimes_AB$ in $Mod(B)$ is preenveloping. Conversely, let $\Gcal$ be a preenveloping torsion class in $Mod(B)$. We have to show that every $A$-module has a left $Gen(T_1)\star f_*(\Gcal)$-approximation. Since $T_1$ is finendo, using \cite[Theorem 1.1]{GT}, it suffices to check that every $A$-module admits a left $f_*(\Gcal)$-approximation. By assumption, every $B$-module has a left $\Gcal$-approximation and, since $\Xcal_B$ is a bireflective subcategory of $Mod(A)$, the result follows.
\end{proof}

\section{Examples}
In this section, we focus on finite dimensional $\mathbb{K}$-algebras. We investigate homological properties of the ring epimorphism associated with a partial silting module.

\begin{example}\label{ex not hom 1}
Let $\Lambda$ be the quotient of the path algebra over $\mathbb{K}$ for the quiver
$$\xymatrix{1\ar@<1ex>[r]^\alpha & 2\ar@<1ex>[l]^\beta}$$
by the ideal generated by $\alpha\beta\alpha$ and $\beta\alpha\beta$. The simple $\Lambda$-module $S_1$ is partial silting with respect to its minimal projective presentation $\sigma:P_2\rightarrow P_1$. In fact, we have that $\Dcal_\sigma=Add(P_1\oplus P_1/rad^2P_1\oplus S_1)$  and the corresponding completion of $S_1$ to a silting module (in the sense of the previous section) is given by $T:=S_1\oplus P_1^{\oplus 2}$. Moreover, the associated bireflective subcategory of $Mod(\Lambda)$ is given by $Add(P_1/rad^2(P_1))$ yielding the ring epimorphism $\Lambda\ra B$ with $B\cong End_A(T)/\langle e_{S_1}\rangle\cong M_2(\mathbb{K})$ (see Theorem \ref{thm2} and Remark \ref{generate small}).  Since the $\Lambda$-module $P_1/rad^2(P_1)$ has infinite projective dimension and it is periodic with respect to the syzygy, there is some $d>1$ with $Ext^d_\Lambda(B,B)\not= 0$. Thus, the ring epimorphism $\Lambda\ra B$ is not homological.
\end{example}

The example above illustrates the fact that ring epimorphisms associated with partial silting modules (even if simple ones) are in general not homological. In the tilting case, however, one can provide sufficient conditions for this property to hold. The following result is motivated by \cite[Theorem 4.16]{GeLe}. For simplicity we use the convention $M^{\oplus 0}:=0$ for any $\Lambda$-module $M$.

\begin{proposition}\label{prop brick}
Let $\Lambda$ be a finite dimensional $\Kbb$-algebra and $T_1$ be a finite dimensional, non-projective, partial tilting $\Lambda$-module with $End_\Lambda(T_1)\cong\Kbb$. Then the associated ring epimorphism $f:\Lambda\ra B$ has kernel $\tau_{T_1}(\Lambda)$, and it is homological if and only if $Ker(f)\cong T_1^{\oplus n}$ for some $n\geq 0$. Moreover, if $f$ is homological, then $B$ has projective dimension at most 2 as a right $\Lambda$-module and, in fact, its projective dimension is less or equal than 1 if and only if $f$ is injective.
\end{proposition}
\begin{proof}
Choose $\Lambda$ and $T_1$ as above and let $\sigma$ be a monomorphic projective presentation of $T_1$. Let $f:\Lambda\ra B$ be the ring epimorphism associated with the bireflective subcategory $\Dcal_\sigma\cap T_1^\circ=T_1^\perp$.
Since $T_1$ is finite dimensional and indecomposable, we can choose a Bongartz complement $T_0$ of $T_1$ such that there is a short exact sequence 
$$\mu_1:\ \ \ \  \xymatrix{0\ar[r]& \Lambda\ar[r]^\phi&T_0\ar[r]^\psi& T_1^{\oplus k}\ar[r]&0,}$$
for some $k\geq 1$, where the map $\phi$ is a minimal left $T_1^{\perp_1}$-approximation of $\Lambda$.  From the construction of the $\Xcal_B$-reflections in $Mod(\Lambda)$ (see Proposition \ref{reflections}), it follows that $Ker(f)\cong \tau_{T_1}(T_0)\cap \phi(\Lambda)$. We will show that $\phi$ induces an isomorphism $\tau_{T_1}(\Lambda)\cong \tau_{T_1}(T_0)$.  For that, it is enough to show that the monomorphism $Hom_\Lambda(T_1,\phi):Hom_\Lambda(T_1,\Lambda)\ra Hom_\Lambda(T_1,T_0)$ is an isomorphism. Given $\eta:T_1\rightarrow T_0$, if the composition $\psi\circ\eta:T_1\rightarrow T_1^{\oplus k}$ is non-zero, then it is a split monomorphism since $End_{\Lambda}(T_1)\cong \Kbb$. Therefore, $\eta$ is a split monomorphism and $\psi_{|\eta(T_1)}$ is an isomorphism between a direct summand of $T_0$ 
 and a direct summand of $T_1^{\oplus k}$, contradicting the minimality of $\phi$. This shows that $\psi\circ\eta=0$ and $\eta$ factors through $\phi$. We conclude that $Ker(f)=\tau_{T_1}(\Lambda)\cong \tau_{T_1}(T_0)$.

Suppose now that $Ker(f)=T_1^{\oplus n}$, for some $n\geq 0$. Since we know that $Tor_1^\Lambda(B,B)=0$, by applying the functor $-\otimes_\Lambda B$ to the sequence
$$\xymatrix{0\ar[r]& T_1^{\oplus n}\ar[r]&T_0\ar[r]& B\ar[r]&0,}$$
we see that $Tor_i^\Lambda(B,B)=0$, for all $i>2$ (since both $T_0$ and $T_1^{\oplus n}$ have projective dimension at most 1). Moreover, since $\sigma\otimes_\Lambda B$ is an isomorphism by Remark \ref{rem hom tensor}, we have that $Tor_1(T_1^{\oplus n},B)=0$ and, thus, $Tor_2^\Lambda(B,B)=0$, proving that $f$ is homological.

Conversely, suppose that $f$ is a homological ring epimorphism and let $n$ be the dimension of the $\Kbb$-vector space $Hom_\Lambda(T_1,\Lambda)$. Choosing a basis of $Hom_\Lambda(T_1,\Lambda)$, let $\epsilon:T_1^{\oplus n}\rightarrow \Lambda$ denote the induced universal map and consider the short exact sequence induced by it
$$\mu_2:\ \ \ \ \xymatrix{0\ar[r]&Ker(\epsilon)\ar[r]&T_1^{\oplus n}\ar[r]^{\epsilon\ \ }&\tau_{T_1}(\Lambda)\ar[r]&0.}$$
It follows, by construction, that $Ker(\epsilon)$ lies in $T_1^{\perp}=\Xcal_B$. We now observe that also $Ker(\epsilon)\otimes_\Lambda B=0$. Applying $-\otimes_\Lambda B$ to the sequence $\mu_1$ we see that $Tor_1^\Lambda(T_0,B)=0$. Applying the same functor to the sequence 
$$\xymatrix{0\ar[r]& \tau_{T_1}(T_0)\ar[r]&T_0\ar[r]&B\ar[r]&0}$$
we conclude that $Tor_1(\tau_{T_1}(T_0),B)=0$, because $f$ is homological. Since $\tau_{T_1}(T_0)\cong\tau_{T_1}(\Lambda)$, by applying once again $-\otimes_\Lambda B$ to the sequence $\mu_2$ we see that $Ker(\epsilon)\otimes_\Lambda B=0$. Since $Ker(\epsilon)$ lies in $\Xcal_B$, this means that $Ker(\epsilon)=0$ and $T_1^{\oplus n}\cong \tau_{T_1}(\Lambda)=Ker(f)$. 

Finally, if $f$ is homological, the previous assertions show that, as a right $\Lambda$-module, it is isomorphic to the quotient $T_0/T_1^{\oplus n}$, from which it follows 
that the projective dimension of $B_\Lambda$ is less or equal than 2. If $n=0$ (i.e., $f$ is injective), then the projective dimension of $B$ is less or equal than 1. Conversely, if the projective dimension of $B$ is less or equal than 1, then $T_1^{\oplus n}$ is either zero or projective. Since $T_1$ is by assumption not projective, it follows that $n=0$ and $B\cong T_0$ as a right $\Lambda$-module, thus finishing the proof.
\end{proof}

Note that the proposition shows in particular that if the above ring epimorphism $f:\Lambda\ra B$ is injective, then it is homological. Examples of injective homological ring epimorphisms occur very naturally in the context of hereditary rings, as we will see in the next section. We finish this section with two examples where $f$ arises from a partial tilting module with trivial endomorphism ring and such that in one case $f$ is homological but not injective and in the other case  $f$ is not homological.

\begin{example}\label{ex not hom 2}
Let $\Lambda$ be the quotient of the path algebra over $\mathbb{K}$ for the quiver
$$
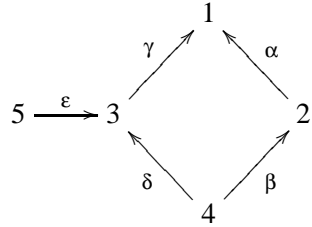
\xymatrix{&&1\\5\ar[r]^\epsilon&3\ar[ur]^\gamma&&2\ar[ul]_\alpha\\&&4\ar[ru]_\beta
\ar[ul]^\delta}$$
by the ideal generated by $\alpha\beta-\gamma\delta$ and $\gamma\epsilon$. The Auslander-Reiten quiver of $\Lambda$ is given by
$$
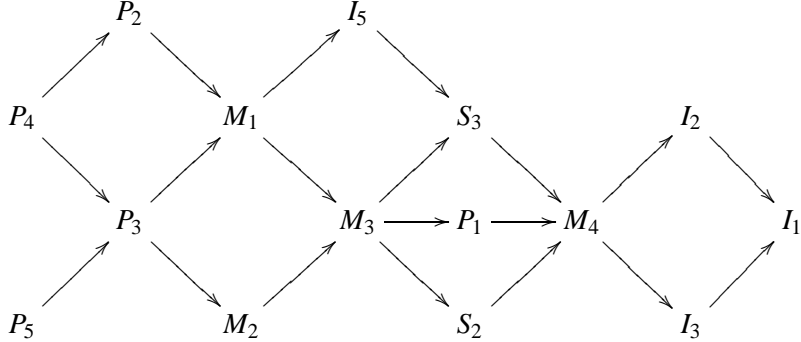
\xymatrix{& P_2\ar[dr] & & I_5\ar[dr] & & & & \\ P_4\ar[ur]\ar[dr] & & M_1\ar[ur]\ar[dr] & & S_3\ar[dr] & & I_2\ar[dr] & \\ & P_3\ar[dr]\ar[ur] & & M_3\ar[ur]\ar[r]\ar[dr] & P_1\ar[r] & M_4\ar[dr]\ar[ur] & & I_1 \\ P_5\ar[ur] & & M_2\ar[ur] & & S_2\ar[ur] & & I_3\ar[ur] &}$$
Given a finite dimensional partial tilting module $\Lambda$-module $T_1$ we denote, as before, by $f:\Lambda\ra B$ the associated ring epimorphism.

\textbf{Case 1: The epimorphism $f$ is homological but not injective.} Let $T_1:=M_2$. The associated bireflective subcategory is described by $T_1^{\perp}=Add(P_2\oplus P_3\oplus P_4\oplus I_2 \oplus M_1\oplus I_5\oplus S_2\oplus I_1)$. It is clear that $\tau_{T_1}(\Lambda)=\tau_{M_2}(P_1)\cong M_2$ and, thus, by Proposition \ref{prop brick}, we conclude that $f$ is homological and non-injective. In particular, the projective dimension of $B$ as an $\Lambda$-module is exactly 2. This can also be computed directly from the description of $T_1^\perp$ by observing that $I_2$ is a projective $B$-module which has projective dimension 2 as a $\Lambda$-module. In fact, since $P_2\oplus P_3\oplus P_4\oplus I_2$ is a projective generator in $\Xcal_B$, it follows that $B$ is Morita equivalent to the quotient of the path algebra of the quiver 
$$\xymatrix{\bullet&\bullet\ar[l]_\mu\ar[r]^\nu&\bullet\ar[r]^\omega&\bullet}$$
by the ideal generated by the path $\omega\nu$.

\textbf{Case 2: The epimorphism $f$ is not homological.} Consider the partial tilting
$\Lambda$-module $M_1$. The associated bireflective subcategory is described by $M_1^\perp=Add(P_2\oplus P_3\oplus P_5\oplus M_2\oplus I_1)$. It is clear, however, that $\tau_{T_1}(\Lambda)=\tau_{T_1}(P_1)\cong M_3$. Therefore, by Proposition \ref{prop brick}, $f$ is not homological. This can also be seen by observing that $B$ is Morita equivalent to $\Kbb Q\times \Kbb\times \Kbb$, where $Q$ is the quiver
$\xymatrix{\bullet\ar[r]&\bullet,}$
and, thus, it is hereditary. However, one checks that $Ext_\Lambda^2(I_1,P_5)\not= 0$, showing that $f$ cannot be homological.
\end{example}

\section{Minimal silting modules over hereditary rings}
In this section, we study silting modules over hereditary rings that turn out to play the role of generalised support tilting modules. Afterwards, we define minimal silting modules. This definition allows us to associate a unique ring epimorphism to every such silting module. We then use this assignment to establish a bijection between minimal silting modules and homological ring epimorphisms. We have the following useful lemma.

\begin{lemma}\label{silting is support tilting}
Let $A$ be a hereditary ring and $\Tcal$ be a subcategory of $Mod(A)$ such that $\Tcal=Add(\Tcal)$. 
If $\phi:A\rightarrow T_0$ is a left $\Tcal$-approximation with $Ext_A^1(T_0,T_0)=0$, then $Ker(\phi)=Ann(\Tcal)$ is a two-sided idempotent ideal.
\end{lemma}
\begin{proof}
It is easy to see that $J:=Ker(\phi)$ equals $Ann(\Tcal)$ 
and, thus, it is a two-sided ideal (and projective both as left and as right $A$-module). Since $A$ is hereditary, ${}^{\perp_1}T_0$ is closed for subobjects. It follows that $Ext^1_A(A/J,T_0)=0$. Applying the functor $Hom_A(-,T_0)$ to the short exact sequence induced by the inclusion of $J$ in $A$ and using the fact that any map from $A$ to $T_0$ factors through $\phi$ (and thus through the quotient $A/J$), we conclude that $Hom_A(J,T_0)=0$. Since $J$ is a projective $A$-module and $\phi$ is a $\Tcal$-approximation with $\Tcal\subseteq Gen(T_0)$, we get 
$Hom_A(J,\Tcal)=0.$
Consider now the monomorphism $\bar{\phi}:A/J\rightarrow T_0$ induced by $\phi$. Applying the functor $-\otimes_AJ$ to the short exact sequence induced by $\bar{\phi}$, since $Tor_1^A(-,J)=0$, there is a monomorphism $A/J\otimes_AJ\rightarrow T_0\otimes_AJ$. Now, let $f:A^{(I)}\rightarrow T_0$ be an epimorphism, for some set $I$. Then it follows that there is a surjection $f\otimes_AJ:J^{(I)}\rightarrow T_0\otimes_AJ.$ 
Since $J$ is projective and $\Tcal=Add(\Tcal)$, $T_0\otimes_AJ$ lies in $Add(T_0)\subseteq \Tcal$ and, therefore, $f\otimes_AJ=0$, which implies that $T_0\otimes_A J=0$. This shows that $A/J\otimes_A J=J/J^2=0$ and, thus, $J$ is idempotent.
\end{proof}

\begin{proposition}\label{prop silting idempotent}
Let $A$ be a hereditary ring.
\begin{enumerate}
\item An $A$-module $T$ is silting if and only if $T$ is tilting over $A/Ann(T)$ and the ideal $Ann(T)$ is idempotent. In other words, silting $A$-modules are support tilting.
\item\cite[Lemma 4.5]{BS} Let $f:A\rightarrow B$ be a homological ring epimorphism. Then the kernel of $f$ is an idempotent ideal. In particular, $f$ can be written as the composition of two homological ring epimorphisms: $A\ra A/Ker(f)$ and $A/Ker(f)\rightarrow B$.
\end{enumerate}
\end{proposition}

\begin{proof}
(1) Assume that $T$ is silting. Thus, by Proposition \ref{tilt silt}(1), $T$ is tilting over the quotient ring  $\bar {A}:=A/Ann(T)$. Moreover, there is a left $Gen(T)$-approximation $\phi:A\ra T_0$ with $T_0$ in $Add(T)$ and $Ker(\phi)=Ann(T)=Ann(Gen(T))$. Since $T$ is silting, $T_0$ has no self-extensions and, thus, by Lemma \ref{silting is support tilting}, $Ann(T)$ is idempotent. Conversely, suppose that $T$ is a tilting $\bar{A}$-module with $Ann(T)$ idempotent. Consider the projective $A$-presentation $\sigma$ of $T$ given as the direct sum of a monomorphic presentation of $T$ with the trivial map $Ann(T)\ra 0$. Since $Ann(T)$ is idempotent, it follows from Proposition \ref{surjective ring epi}(2) that
$$\Dcal_\sigma=T^{\perp_1}\cap Ann(T)^\circ=Ker(Ext_{\bar{A}}^1(T,-))=Gen(T).$$
Consequently, $T$ is a silting $A$-module.
\end{proof}

Note that a similar statement does not hold without the hereditary assumption.

\begin{example}
Let $T$ be a sincere finitely generated silting module over a finite dimensional $\mathbb{K}$-algebra $\Lambda$ that is not tilting. Such modules $T$ are just non-faithful $\tau$-tilting modules over $\Lambda$ (see \cite{AMV} and \cite{AIR}). Since $T$ is not faithful, $Ann(T)\not= 0$ and since $T$ is sincere, $Ann(T)$ cannot contain any idempotent $e\not= 0$ of $\Lambda$. In particular, it is not an idempotent ideal. Moreover, Example \ref{ex not hom 2} (Case 1) provides an example of a homological ring epimorphism whose kernel is not idempotent.
\end{example}

In the following, we wish to assign a ring epimorphism to a silting module $T$, using the construction of Section 3. To this end, we need a canonical choice of a partial silting module $T_1$ associated with $T$. Therefore we consider the following class of silting modules.

\begin{definition}
Let $A$ be a hereditary ring and $T$ be a silting $A$-module. Then $T$ is called \textbf{minimal}, if $A_A$ admits a minimal left $Add(T)$-approximation.
\end{definition}

Clearly, the definition of minimal silting modules also applies to tilting modules. Note that already in the setting of tilting modules, we obtain many non-trivial examples.

\begin{example}\label{example minimal silting}
Let $A$ be a hereditary ring.
\begin{enumerate}
\item Let $T$ be an endofinite silting $A$-module, i.e., $T$ has finite length over its endomorphism ring. Therefore, by \cite[Theorem 4.1]{KS}, $Add(T)$ is closed for products and, thus, by \cite[Theorem 3.1]{KS}, every $A$-module admits a minimal left $Add(T)$-approximation. In particular, finitely generated silting modules over hereditary Artin algebras are minimal. 
\item Let $A$ be noetherian and consider the minimal injective coresolution of the free module of rank one
$$\xymatrix{0\ar[r] & A\ar[r] & E_1\ar[r] & E_2\ar[r] & 0.}$$
It follows that $T:=E_1\oplus E_2$ is a tilting $A$-module where $Gen(T)$ is given by the class of injective $A$-modules. Since injective envelopes are left-minimal, $T$ is a minimal tilting module.
\item Let $T$ be a tilting $A$-module that arises from a ring epimorphism, say $T=B\oplus B/A\,$ for an injective homological ring epimorphism $f:A\ra B$ (see Theorem \ref{ring epis and tilting}(1)). Then $T$ is minimal. In fact, we have the following canonical $Add(T)$-approximation sequence
$$\xymatrix{0\ar[r] & A\ar[r]^f & B\ar[r] & B/A\ar[r] & 0.}$$
Since $f$ is a reflection map, it is clearly left-minimal. 
\end{enumerate}
\end{example}

Minimal silting modules are motivated by the following construction. Let $A$ be hereditary, and let $T$ be a minimal silting $A$-module with associated torsion class $\Dcal_\sigma=Gen(T)$. Consider the minimal $Add(T)$-approximation sequence
$$\xymatrix{A\ar[r]^\phi & T_0\ar[r] & T_1\ar[r] & 0.}$$
By Proposition \ref{prop silting idempotent}(1), $T$ is a tilting module over the quotient ring $\bar{A}:=A/Ann(T)$ and $Ann(T)$ is idempotent. We get the induced minimal $Add(T)$-approximation sequence in $Mod(\bar{A})$
$$\xymatrix{0\ar[r] & \bar{A}\ar[r]^\phi & T_0\ar[r] & T_1\ar[r] & 0.}$$
Since the ideal $Ann(T)$ is idempotent we get from Proposition \ref{surjective ring epi}(2) that
$$Gen(T)=T_1^{\perp_1}\cap\Xcal_{\bar{A}}=T_1^{\perp_1}\cap Ann(T)^\circ.$$
In particular, $T_1$ is a partial silting $A$-module with respect to the projective presentation $\sigma_1$, given as the direct sum of a monomorphic presentation of $T_1$ with the trivial map $Ann(T)\ra 0$. In fact, we get the equality $\Dcal_{\sigma_1}=Gen(T)=\Dcal_\sigma$. Following Proposition \ref{prop bireflective}, we consider the bireflective subcategory
$$\Ycal=T_1^{\perp}\cap\Xcal_{\bar{A}}$$
associated with $T_1$. The corresponding ring epimorphism will be denoted by $A\ra B_T$. Since the approximation $\phi$ was chosen minimal and, hence, the module $T_1$ is uniquely determined, we obtain a well-defined map from (equivalence classes of) minimal silting modules to (epiclasses of) ring epimorphisms by mapping
$T$ to the ring epimorphism $A\ra B_T$. We need the following technical proposition motivated by the results in \cite[Section 2]{IT}.

\begin{proposition}\label{technical proposition}
Let $A$ be a hereditary ring and $T$ be a minimal silting $A$-module with associated ring epimorphism $A\ra B_T$. Then $\Xcal_{B_T}$ coincides with
$${\mathfrak a}(Gen(T)):=\{X\in Gen(T)\mid\forall(g:Y\ra X)\in Gen(T),\, Ker(g)\in Gen(T)\}.$$
Moreover, if $A\ra T_0$ is the minimal left $Add(T)$-approximation, then we have $Proj(\Xcal_{B_T})=Add(T_0)$. 
\end{proposition}

\begin{proof} 
Let $\phi:A\rightarrow T_0$ be the minimal left $Add(T)$-approximation and $T_1=Coker(\phi)$. Let $\bar{A}=A/Ann(T)$ and observe that, by Lemma \ref{CTT compare}, we have
$$\Xcal_{B_T}=\{X\in\Xcal_{\bar{A}}\mid Hom_{\bar{A}}(T_1,X)=0=Ext_{\bar{A}}^1(T_1,X)\}.$$

We first prove that $\Xcal_{B_T}\subseteq{\mathfrak a}(Gen(T))$. Take $X$ in $\Xcal_{B_T}$. Since $Ext_{\bar{A}}^1(T_1,X)=0$ and $T$ is a tilting module over $\bar{A}$, the module $X$ lies in $Gen(T)$. Now consider a test map $g:Y\rightarrow X$ with $Y$ in $Gen(T)$. Without loss of generality, we may assume $g$ to be surjective, since $Im(g)$ also lies in $\Xcal_{B_T}$. Moreover, note that $Ker(g)$ belongs to $\Xcal_{\bar{A}}$, since so do $X$ and $Y$. By applying the functor $Hom_{\bar{A}}(T_1,-)$ to the short exact sequence induced by $g$, we obtain the exact sequence
$$\xymatrix{Hom_{\bar{A}}(T_1,Y)\ar[r] & Hom_{\bar{A}}(T_1,X)\ar[r] & Ext_{\bar{A}}^1(T_1,Ker(g))\ar[r] & Ext_{\bar{A}}^1(T_1,Y).}$$
Since, by assumption, $Hom_{\bar{A}}(T_1,X)=0$ and $Y$ lies in $Gen(T)$ (showing that $Ext_{\bar{A}}^1(T_1,Y)=0$), it follows that $Ext_{\bar{A}}^1(T_1,Ker(g))=0$. This proves that $Ker(g)$ lies in $Gen(T)$ and, thus, $X$ lies in ${\mathfrak a}(Gen(T))$.

Conversely, since ${\mathfrak a}(Gen(T))\subseteq Gen(T)$, it is enough to show that $Hom_{\bar{A}}(T_1,{\mathfrak a}(Gen(T)))=0$. By definition, ${\mathfrak a}(Gen(T))$ is closed for subobjects in $Gen(T)$. In particular, the image of any morphism from an object in $Gen(T)$ to an object in ${\mathfrak a}(Gen(T))$ is itself in ${\mathfrak a}(Gen(T))$. Thus, to prove our claim, it is enough to show that there are no surjections from $T_1$ to any object $C$ in ${\mathfrak a}(Gen(T))$. Let $\omega:T_1\rightarrow C$ be such a surjection and consider the commutative diagram
$$\xymatrix{& A\ar[d]^d\ar[r]^1 & A\ar[d]^\phi\\ 0\ar[r]& K\ar[r]^a\ar[d]^b&T_0\ar[r]^{\omega\circ \psi}\ar[d]^\psi&C\ar[r]\ar[d]^1&0\\ 0\ar[r]& X\ar[d]\ar[r]^c&T_1\ar[d]\ar[r]^\omega&C\ar[r]&0\\ & 0& 0}$$
with exact rows and columns. Since $C$ lies in ${\mathfrak a}(Gen(T))$, $K$ lies in $Gen(T)$. Thus, $d$ factors through $\phi$, i.e., there is a map $e:T_0\rightarrow K$ such that $e\circ\phi = d$. This shows that $a\circ e\circ \phi=a\circ d=\phi$ which by minimality of $\phi$ yields that $a\circ e$ is an isomorphism. In particular, $a$ is an epimorphism and $C=0$.

Finally, let us show that $Proj(\Xcal_{B_T})=Add(T_0)$ by observing that $T_0$ is a projective generator in ${\mathfrak a}(Gen(T))$. Since $A$ is hereditary, by \cite[Proposition 2.15]{IT},  ${\mathfrak a}(Gen(T))$  coincides with
$$\{X\in Gen(T)\mid\forall(g:Y\twoheadrightarrow X)\in Gen(T),\, Ker(g)\in Gen(T)\}.$$
First we see that $T_0$ lies in ${\mathfrak a}(Gen(T))$. Given $X$ in $Gen(T)$ and an epimorphism $\omega:X\rightarrow T_0$, since $A$ is projective there is $\beta:A\rightarrow X$ such that $\omega\circ\beta=\phi$. Moreover, since $T_0$ is a left $Gen(T)$-approximation, there is $\gamma:T_0\rightarrow X$ such that $\beta=\gamma\circ\phi$. Therefore, we have that $\omega\circ\gamma\circ\phi=\omega\circ\beta=\phi$, which by minimality of $\phi$ shows that $\omega$ is a split epimorphism and, thus, $Ker(\omega)$ lies in $Gen(T)$. Now, $T_0$ is projective in ${\mathfrak a}(Gen(T))$ since $Hom_A(T_0,-)$ is exact for short exact sequences $Gen(T)$. Moreover, it is a generator of $Gen(T)$ and, hence, also a generator for ${\mathfrak a}(Gen(T))$, thus proving our claim.
\end{proof}

\begin{remark}
The first part of the above proof shows the following. Let $A$ be any ring and $T$ be a silting $A$-module such that $A$ admits a minimal left $Gen(T)$-approximation $\phi:A\rightarrow T_0$. Then the bireflective subcategory $Gen(T)\cap Coker(\phi)^\circ$ coincides with ${\mathfrak a}(Gen(T))$. However, notice that, in contrast to the hereditary setting, $Hom_A(Coker(\phi),T_0)$ may not vanish and $T_0$ does not generally lie in ${\mathfrak a}(Gen(T))$.
\end{remark}

Now we are able to state the main result of this section.

\begin{theorem}\label{main silting epi}
Let $A$ be a hereditary ring. Then the assignment $\alpha:T\mapsto(f:A\ra B_T)$ yields a bijection between 
\begin{enumerate}
\item equivalence classes of minimal silting $A$-modules;
\item epiclasses of homological ring epimorphisms of $A$.
\end{enumerate}
Moreover, $\alpha$ restricts to a bijection between
\begin{enumerate}
\item equivalence classes of minimal tilting $A$-modules;
\item epiclasses of injective homological ring epimorphisms of $A$.
\end{enumerate}
\end{theorem}
\begin{proof}
First, we check that $\alpha$ is well-defined. By construction, the ring epimorphism $f$ is uniquely determined by $T$. Moreover, by Proposition \ref{prop bireflective}, the subcategory $\Xcal_{B_T}$ is closed under extensions in $Mod(A)$ and, therefore, by Proposition \ref{Tor Ext}, $Tor_1^A(B_T,B_T)=0$. Since $A$ is hereditary, all higher $Tor$-groups vanish showing that the ring epimorphism $f$ is homological. 

The injectivity of $\alpha$ follows from Proposition \ref{technical proposition}: if $T$ and $T'$ are minimal silting modules with $\alpha(T)=\alpha(T')$, then we have that $Add(T_0)=Add(T_0')$ and, thus, $T$ and $T'$ are equivalent.

Next, we prove the surjectivity of $\alpha$. Let $f:A\ra B$ be a homological ring epimorphism. By Proposition \ref{prop silting idempotent}(2), we get a commutative diagram of homological ring epimorphisms\vspace{-0.2cm}
$$\xymatrix{A\ar[rr]^f\ar[dr] & & B\\ & A/Ker(f)\ar[ur]^{f'} & }$$
where $f'$ is injective and the quotient ring $\bar{A}:=A/Ker(f)$ is again hereditary. Thus, by Theorem \ref{ring epis and tilting}(1), $T:=B\oplus B/\bar{A}$ is a tilting module over $\bar{A}$ and it is minimal as argued in Example \ref{example minimal silting}(3). By Proposition \ref{prop silting idempotent}(1), $T$ becomes a minimal silting $A$-module with respect to a projective $A$-presentation $\sigma$ of $T$ that is given as the direct sum of a monomorphic presentation of $T$ and the trivial map $Ker(f)\ra 0$.
It remains to check that $\alpha(T)$ lies in the same epiclass as the ring epimorphism $f:A\ra B$. Since the minimal left $Add(T)$-approximation of $A$ is given by the module map $f:A\ra B$, by construction, we have
$$\Xcal_{B_T}=(B/\bar{A})^\perp\cap\Xcal_{\bar{A}}.$$ 
But this coincides with $\Xcal_B$ by Theorem \ref{ring epis and tilting}(1).

It follows from the previous arguments that the inverse $\alpha^{-1}$ of $\alpha$ assigns to a homological ring epimorphism $f:A\ra B$ the minimal silting $A$-module $B\oplus Coker(f)$. Therefore, in case $f$ is injective, the module $\alpha^{-1}(f)$ is actually a tilting module, by Theorem \ref{ring epis and tilting}(1). It remains to check the restriction of the map $\alpha$. Let $T$ be a minimal tilting $A$-module with a monomorphic minimal $Add(T)$-approximation $\phi:A\ra T_0$. Using Proposition \ref{technical proposition}, it follows that $T_0$ lies in $\Xcal_{B_T}$ and, thus, $\phi$ is the $\Xcal_{B_T}$-reflection of $A$. In particular, the ring epimorphism $A\ra B_T$ is injective.
\end{proof}

We have the following immediate corollary of Theorem \ref{main silting epi}.

\begin{corollary}
Let $A$ be a hereditary ring. Then a tilting $A$-module $T$ arises from a ring epimorphism if and only if $T$ is minimal.
\end{corollary}

\begin{example}\label{Kronecker}
Let $\Lambda$ be the Kronecker algebra, i.e., the path algebra of the quiver $\xy\xymatrixcolsep{2pc}\xymatrix{ \bullet \ar@<0.5ex>[r]  \ar@<-0.5ex>[r] & \bullet } \endxy$ over a field $\Kbb$. First observe that a non-zero support tilting module which is not tilting is equivalent to a simple $\Lambda$-module. These are clearly minimal silting modules. From \cite{Ma} and \cite{AS2} we know that all tilting modules except the Lukas tilting module (see \cite{L1} and \cite{L2}) arise from ring epimorphisms. By our result this is a classification of all minimal silting $\Lambda$-modules and the unique non-minimal silting module is the Lukas tilting module.

Let us briefly analyse the Lukas tilting module $L$ in more detail. 
The tilting class $Gen(L)$ is given by the $\Lambda$-modules without any indecomposable preprojective summands. Since $Add(L')=Add(L)$ for all non-zero direct summands $L'$ of $L$ (see \cite[Theorem 6.1]{L1} and \cite[Theorem 3.1]{L2}), we conclude that for all $Add(L)$-approximation sequences
$$\xymatrix{0\ar[r] & \Lambda\ar[r] & L_0\ar[r] & L_1\ar[r] & 0}$$
the bireflective subcategory $L_1^\perp$ only contains the zero-module. 
\end{example}

Following \cite{IT}, we refer to full, abelian and extension closed subcategories of $mod(A)$  as \textbf{wide}, and we call them  \textbf{finitely generated} if they contain a generator.
\begin{corollary}\cite[Section 2]{IT},\cite[Theorem 4.2]{Ma}\label{IngallsThomas} If $\Lambda$ is a finite dimensional hereditary algebra, the map $\alpha$ from Theorem \ref{main silting epi} restricts to a bijection between 
\begin{enumerate}
\item equivalence classes of finitely generated support tilting {$\Lambda$-}modules;
\item epiclasses of homological ring epimorphisms  $\Lambda\to B$ with $B$ finite dimensional;
\end{enumerate}
and there is a further bijection with
\begin{enumerate}
\item[(3)] finitely generated wide subcategories of $mod(\Lambda)$
\end{enumerate}
by assigning to a ring epimorphism $\Lambda\to B$ the class $\mathcal X_B\cap mod(\Lambda)\cong mod(B)$.
\end{corollary}

In \cite{IT,R} further bijections are established, providing a combinatorial interpretation of finitely generated support tilting modules in terms of noncrossing partitions,   clusters, or antichains.
The  poset of exceptional antichains, which is isomorphic to the poset of generalised noncrossing partitions (\cite[Theorem 3.6.6]{R}), is  defined via the partial order on finitely generated wide subcategories given by inclusion. 
By Corollary \ref{IngallsThomas},
the latter corresponds to the following  partial order on the  epiclasses of homological ring epimorphisms of $A$. Given $f_1:A\to B_1$ and $f_2:A\to B_2$, we set $$f_1\ge f_2$$ if there is a ring homomorphism $g:B_1\to B_2$ such that $g\circ f_1=f_2$ (it is indeed a partial order because ring epimorphisms are left minimal maps). 
In order to investigate $\le$, we first need some terminology.

\begin{theorem}{\cite[Theorem~4.1]{Sch}}\label{def:universallocalisation}
Let $A$ be a ring and $\Sigma$ be a class of morphisms between
finitely generated projective right $A$-modules. Then
there is a ring homomorphism 
$f: A\rightarrow A_\Sigma$ such that
\begin{enumerate}
\item $f$ is \emph{$\Sigma$-inverting,} i.e.~if
$\sigma$ belongs to  $\Sigma$, then
$\sigma\otimes_A A_\Sigma$ is an isomorphism of right
$A_\Sigma$-modules, and \item $f$ is \emph{universal
$\Sigma$-inverting}, i.e.~for any $\Sigma$-inverting morphism $f': A\rightarrow B$
there exists a unique ring homomorphism ${g}:
A_\Sigma\rightarrow B$ such that $g\circ f=f'$.
\end{enumerate}
The homomorphism
$f\colon A\rightarrow A_\Sigma$ is a ring epimorphism
with  $Tor_1^{A}(A_\Sigma,{A_\Sigma})=0$, called
the \textbf{universal localisation} of $A$ at
$\Sigma$.
\end{theorem}

Let now ${\mathcal{U}}$ be a set of finitely presented modules of projective dimension at most one.
For each $U\in\mathcal{U},$ we fix a  projective resolution 
$0\to P\stackrel{{\sigma_U}}{\to} Q\to U\to 0$  in $mod(A)$
and we set $\Sigma=\{\sigma_U\mid U\in\mathcal{U}\}.$ 
We denote by {$f_{\mathcal U}:A\to A_{\mathcal{U}}$} the universal localisation of $A$ at
 $\Sigma$. Note that  $A_{\mathcal
U}$ does not depend on the chosen class $\Sigma$ (compare \cite[Theorem~0.6.2]{Cohn}).
Moreover, the ring epimorphism $f_{\mathcal U}$ corresponds to the bireflective subcategory $\mathcal X_{ A_{\mathcal U}}=\mathcal U^\perp$ by \cite[Proposition 2.7]{AA}.
Finally, if $A$ is hereditary, then $f_{\mathcal U}$ is injective if and only if the modules in $\mathcal U$ are \textbf{bound}, i.e.~they are finitely presented modules $U$ such that $Hom_A(U,A)=0$ (compare \cite{Scho2} and \cite[Lemma 4.1]{Ma}). 

\begin{theorem}\cite[Theorem 2.3]{Scho2},\cite[Theorem 6.1]{KSt} \label{univlochom}
Let $A$ be a hereditary ring. Then a ring epimorphism starting in $A$ is homological if and only if it is a universal localisation. 
Moreover, the assignment $\gamma: \mathcal W\mapsto (f_{\mathcal W}: A\to A_{\mathcal W})$ defines a bijection between
\begin{enumerate}
\item  wide subcategories of $mod(A)$;
\item epiclasses of universal localisations of $A$
\end{enumerate}
which restricts  to a bijection between
\begin{enumerate}
\item  wide subcategories of bound $A$-modules;
\item epiclasses of injective universal localisations of $A$.
\end{enumerate}
\end{theorem} 

\begin{remark}\label{wide}
(1) Notice the difference between the bijections in Corollary \ref{IngallsThomas} and Theorem \ref{univlochom}: in the first case the wide subcategory associated with $f:A\to B$ is $\mathcal X_B\cap mod(A)$, while in the second case we are taking the wide subcategory $\mathcal W$ of all modules $U$ with a projective resolution $\sigma_U$ which is inverted by the functor $-\otimes_A B$, or in other words, $\mathcal W = {}^\perp\mathcal X_B\cap mod(A)$, and $\mathcal X_B=\Wcal^\perp$.

(2)  By \cite[Theorem 2.5]{Scho2}, the wide subcategories $\mathcal W$ consisting of bound $A$-modules correspond bijectively to  Hom-perpendicular sets of finitely presented bound $A$-modules, that is, antichains of non-projective modules in the terminology of \cite{R}.
\end{remark}

\begin{corollary}\label{triangle}
Let $A$ be a hereditary ring. Then there is a commutative triangle of bijections
$$\xymatrix{{\left\{\begin{array}{c}\text{equivalence classes of} \\ \text{minimal silting }A\text{-modules} \end{array}\right\}}\ar[rr]^{\alpha} &  & {\left\{\begin{array}{c}\text{epiclasses of universal} \\ \text{localisations of } A \end{array}\right\}}\\ & {\left\{\begin{array}{c}\text{wide subcategories } \\ \text{of } mod(A)\ar[ur]^{\gamma}\ar@{-->}[ul]_{\beta}\end{array}\right\}} &}$$
where $\alpha$ is defined in Theorem \ref{main silting epi}, $\gamma$ is defined in Theorem \ref{univlochom} and  $\beta$ assigns to a wide subcategory $\Wcal$ the silting class 
$\Wcal^{\perp_1}\cap\mathcal X_{\bar{A}},\text{ where }\bar{A}=A/ Ker(f_{\Wcal}).$
\end{corollary}
\begin{proof} We only have to show that $\beta=\alpha^{-1}\circ \gamma$. From \cite[Corollary 4.13]{AS1} and \cite[Theorem 2.6]{Scho2} we know that $A_\Wcal \oplus A_\Wcal/\bar{A}$ is a tilting $\bar{A}$-module  with tilting class $Gen(A_\Wcal)= Ker Ext^1_{\bar{A}}(\Wcal,-)=\Wcal^{\perp_1}\cap\mathcal X_{\bar{A}}$. Hence, we have $\beta(\Wcal)=Gen(A_\Wcal)$, which is the silting class of the silting module $A_\Wcal\oplus Coker f_\Wcal= \alpha^{-1}( f_{\Wcal})=\alpha^{-1}( \gamma(\Wcal))$ (see the proof of Theorem \ref{main silting epi}).
\end{proof}

\begin{corollary}\label{lattice}
Let $A$ be a hereditary ring. Then the epiclasses of homological ring epimorphisms starting in $A$ form a lattice with respect to $\le$.
\end{corollary}
\begin{proof}
Let $f_1:A\to B_1$ and $f_2:A\to B_2$ be homological ring epimorphisms. Of course, $f_1\ge f_2$ if and only if $\mathcal X_{B_2}$ is a full subcategory of  $\mathcal X_{B_1}$.  By Theorem \ref{univlochom} and Remark \ref{wide}(1) this amounts to the inclusion $\Wcal_1\subset \Wcal_2$ for the wide subcategories $\Wcal_i=\gamma^{-1}(f_i)$. The claim now follows from the fact that the wide subcategories of $mod(A)$ form a lattice. More precisely, the join of $f_1$ and $f_2$ is given by the universal localisation at the wide subcategory $\Wcal_1\cap\Wcal_2$, and the meet is given by the universal localisation at $\Wcal_1\cup\Wcal_2$ (or equivalently, by Theorem \ref{univlochom}, at the smallest wide subcategory containing $\Wcal_1\cup\Wcal_2$).
\end{proof}

If we restrict to  homological ring epimorphisms with finite dimensional target, we obtain the poset of exceptional antichains. This is known to be a lattice for finite dimensional hereditary algebras of finite representation type (compare \cite{IT,R}). Indeed, this  also follows from Corollary \ref{lattice} since every ring epimorphism has a finite dimensional target in the representation-finite case (\cite[Corollary 2.3]{GdP}). In general, however, the poset of exceptional antichains is not a lattice, as remarked in  \cite[p.65]{R} and illustrated in the following example.

\begin{example} The meet of two homological ring epimorphisms $f_1:A\to B_1$ and $f_2:A\to B_2$ can be infinite dimensional even when $B_1$ and $B_2$ are finite dimensional algebras. Indeed, if $\Lambda$ is a finite dimensional tame hereditary algebra and $\Wcal$ is a non-homogeneous tube with simple regular modules $S_1,\ldots,S_r$, then $f_1:\Lambda\to \Lambda_{\{S_1\}}$ and $f_2: \Lambda\to  \Lambda_{\{S_2,\ldots,S_r\}}$ have finite dimensional targets, but their meet $f: \Lambda\to  \Lambda_\Wcal$ has an infinite dimensional target (see \cite[Section 4]{CB} and \cite[Proposition 1.10]{AS2}). A specific instance of this phenomenon is provided in \cite[Example 3.1.4]{R}.
\end{example}

\begin{example} We compute the lattice of homological ring epimorphisms for the Kronecker algebra $\Lambda$ (see Example \ref{Kronecker}). Let us denote by $P_i$ (respectively $Q_i$), with $i\in\mathbb{N}$, the (finite dimensional) indecomposable preprojective (respectively, preinjective) modules, indexed such that $dim_\Kbb Hom_A(P_i,P_{i+1})=2$ (respectively, $dim_\Kbb Hom_A(Q_{i+1},Q_i)=2$). Also, we identify below the quasi-simple regular $\Lambda$-modules with points in the projective line $\mathbb{P}^1_\Kbb$. Following Example \ref{Kronecker}, we can list all minimal silting $\Lambda$-modules and all homological ring epimorphisms of $\Lambda$ (together with their associated bireflective subcategories) as follows. 
$$\begin{array}{ccccc}
\text{Silting module} && \text{Homological ring epimorphism} && \text{Bireflective subcategory of }Mod(\Lambda)\\
\hline
0&& \Lambda\rightarrow 0 && \Xcal_0=0\\
\Lambda=P_1\oplus P_2 && Id && \Xcal_{Id}=Mod(\Lambda)\\
P_1 && \lambda_0:\Lambda\rightarrow \Lambda/\Lambda e_2\Lambda&& \Xcal_{\lambda_0}=Add(P_1) \\
Q_1 && \mu_0:\Lambda\rightarrow \Lambda/\Lambda e_1\Lambda && \Xcal_{\mu_0}=Add(Q_1)\\
(P_i\oplus P_{i+1})_{i\geq 2} && (\lambda_i:\Lambda\rightarrow \Lambda_{\{P_{i+1}\}})_{i\in\mathbb{N}} && (\Xcal_{\lambda_i}=Add(P_i))_{i\in\mathbb{N}}\\
(Q_{i+1}\oplus Q_{i})_{i\in\mathbb{N}} && (\mu_i:\Lambda\rightarrow \Lambda_{\{Q_{i}\}})_{i\in\mathbb{N}} && (\Xcal_{\mu_i}=Add(Q_{i+1}))_{i\in\mathbb{N}}\\
(\Lambda_\Ucal\oplus \Lambda_\Ucal/\Lambda)_{\emptyset\neq \Ucal\subseteq \mathbb{P}^1_\Kbb}&& (\lambda_\Ucal:\Lambda\rightarrow \Lambda_\Ucal)_{\emptyset\neq \Ucal\subseteq \mathbb{P}^1_\Kbb} && (\Xcal_{\lambda_\Ucal}=\Ucal^\perp)_{\emptyset\neq \Ucal\subseteq \mathbb{P}^1_\Kbb}
\end{array}$$

The lattice of homological ring epimorphisms for the Kronecker algebra is then as follows
$$
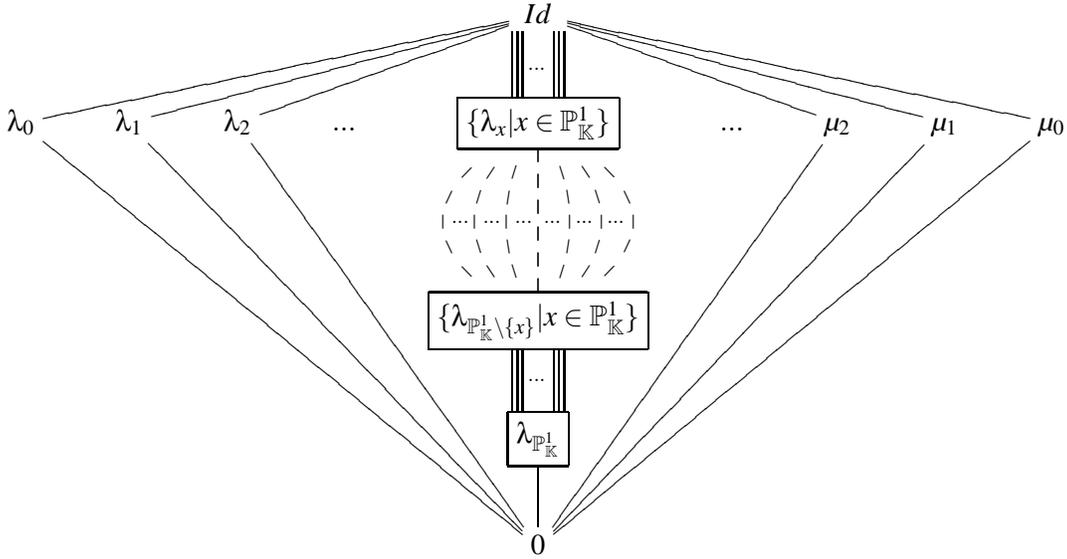
\xymatrix{&&&& \ Id \ \ar@{-}[lllld]\ar@{-}[llld]\ar@{-}[lld]\ar@{-}[rrrrd]\ar@{-}[rrrd]\ar@{-}[rrd]\ar@<-2ex>@{-}[d]^{\ ...}\ar@<-1.6ex>@{-}[d]\ar@<-1.2ex>@{-}[d]\ar@<1.2ex>@{-}[d]\ar@<1.6ex>@{-}[d]\ar@<2ex>@{-}[d]\\ \lambda_0\ar@{-}[rrrrdddd]&\lambda_1\ar@{-}[rrrdddd]&\lambda_2\ar@{-}[rrdddd] & ... & *+[F]{\{\lambda_x| x\in \mathbb{P}^1_\mathbb{K}\}}\ar@{--}@/_1pc/[dd]^{...}\ar@{--}[dd]\ar@{--}@/_2pc/[dd]^{...}\ar@{--}@/_3pc/[dd]^{...}\ar@{--}[dd]\ar@{--}@/^3pc/[dd]_{...}\ar@{--}@/^2pc/[dd]_{...}\ar@{--}@/^1pc/[dd]_{...} & ... & \mu_2\ar@{-}[lldddd] & \mu_1\ar@{-}[llldddd] &\mu_0\ar@{-}[lllldddd] \\ \\ &&&& *+[F]{\{\lambda_{\mathbb{P}^1_\Kbb\setminus \{x\}}| x\in \mathbb{P}^1_\mathbb{K}\}}\ar@<-2ex>@{-}[d]^{\ ...}\ar@<-1.6ex>@{-}[d]\ar@<-1.2ex>@{-}[d]\ar@<1.2ex>@{-}[d]\ar@<1.6ex>@{-}[d]\ar@<2ex>@{-}[d]\\ &&&& *+[F]{\lambda_{\mathbb{P}^1_\Kbb}}\ar@{-}[d]\\ &&&&0}$$
where the interval between $Id$ and $\lambda_{\mathbb{P}^1_\Kbb}$ represents the dual poset of subsets of $\mathbb{P}^1_\Kbb$. The ring epimorphisms with infinite dimensional target are those in frames, i.e., those of the form $\lambda_\Ucal$ with $\emptyset\neq\Ucal\subseteq \mathbb{P}^1_\Kbb$. The poset obtained by excluding these elements is precisely the poset of exceptional antichains from \cite{R}. Note that the above lattice is dual to the one of wide subcategories in $mod(\Lambda)$. The poset of silting classes is, however, completely different. As before, let $L$ denote the Lukas tilting module (which is not minimal and, thus, does not appear above). 
The silting classes corresponding to infinite dimensional silting modules are framed.
$$
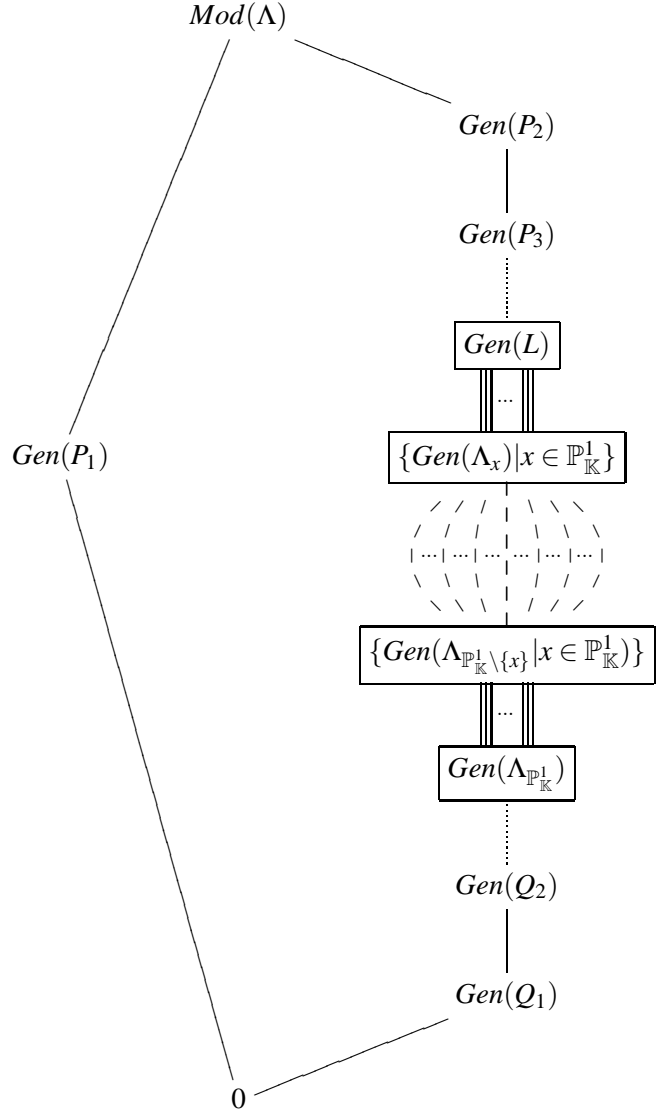
\xymatrix{&Mod(\Lambda)\ar@{-}[ddddl]\ar@{-}[dr]&\\ && Gen(P_2)\ar@{-}[d]\\ && Gen(P_3)\ar@{.}[d]\\ && *+[F]{Gen(L)}\ar@<-2ex>@{-}[d]^{\ ...}\ar@<-1.6ex>@{-}[d]\ar@<-1.2ex>@{-}[d]\ar@<1.2ex>@{-}[d]\ar@<1.6ex>@{-}[d]\ar@<2ex>@{-}[d]\\  Gen(P_1)\ar@{-}[ddddddr]&& *+[F]{\{Gen(\Lambda_x)| x\in \mathbb{P}^1_\Kbb\}}\ar@{--}@/_1pc/[dd]^{...}\ar@{--}[dd]\ar@{--}@/_2pc/[dd]^{...}\ar@{--}@/_3pc/[dd]^{...}\ar@{--}[dd]\ar@{--}@/^3pc/[dd]_{...}\ar@{--}@/^2pc/[dd]_{...}\ar@{--}@/^1pc/[dd]_{...}\\  \\&& *+[F]{\{Gen(\Lambda_{\mathbb{P}^1_\Kbb\setminus \{x\}}|x\in\mathbb{P}^1_\Kbb)\}}\ar@<-2ex>@{-}[d]^{\ ...}\ar@<-1.6ex>@{-}[d]\ar@<-1.2ex>@{-}[d]\ar@<1.2ex>@{-}[d]\ar@<1.6ex>@{-}[d]\ar@<2ex>@{-}[d]\\ && *+[F]{Gen(\Lambda_{\mathbb{P}^1_\Kbb})}\ar@{.}[d] \\  && Gen(Q_2)\ar@{-}[d]\\ && Gen(Q_1)\ar@{-}[dl]\\ &0&&}$$
\end{example}

Next, we focus on the case of a Dedekind  (i.e.~commutative and hereditary) domain. Here there is an interesting connection with Gabriel topologies. Indeed, over any coherent ring $A$, there is a bijective correspondence assigning to every Serre subcategory $\mathcal U$ of $mod(A)$ a Gabriel topology of finite type ${\mathcal{L}}_{\mathcal U}$ on $A$ (see \cite[Theorem VI.5.1]{Ste},\cite[Theorem 2.8]{He},\cite[Corollary 2.10]{K}).  Notice that every Serre subcategory is wide, and over commutative noetherian rings also the converse is true by \cite[Theorem A]{Tak}. Also, recall that every Gabriel topology $\mathcal L$ induces a ring homomorphism $A\to Q_{\mathcal L}$. We say that a ring epimorphism $A\to B$ is non-trivial if $B\not=0$, and we consider only non-trivial Gabriel topologies, i.e. consisting of non-zero ideals.
\begin{example}\label{Ded} If $A$ is a {Dedekind domain}, all non-trivial homological ring epimorphism are injective and the maps $\alpha$, $\beta$ and $\gamma$ from Corollary \ref{triangle} define bijections between:
\begin{enumerate}
\item[(a)] equivalence classes of tilting $A$-modules; 
\item[(b)] wide subcategories of bound $A$-modules;
\item[(c)] Gabriel topologies on $A$;
\item[(d)] epiclasses of non-trivial universal localisations of $A$;
\item[(e)] epiclasses of non-trivial homological ring epimorphisms of $A$.
\end{enumerate}
Indeed,  we know that all tilting modules are minimal from \cite[Corollary 6.12]{AS1}.
Further, all Gabriel topologies are of finite type, and over commutative semihereditary rings the latter coincide with perfect Gabriel topologies  by
 \cite[Chapter XI, Proposition 3.3]{Ste}. This means that the localisations $A\to Q_{\mathcal L}$ of $A$ induced by  Gabriel topologies are (flat) ring epimorphisms, and in fact, they are precisely  the non-trivial universal localisations of $A$ by \cite[Theorem 7.8]{BS}.
  So, it only remains to check that every non-trivial universal localisation $f_\Wcal$ is injective. 
But this is true over any commutative semihereditary domain. Indeed, if $\Wcal$ contains a non-bound module, then it contains a non-zero  projective module $P$, which must vanish under $-\otimes_A A_\Wcal$. Then also the trace  of $P$ in $A$ vanishes under $-\otimes_A A_\Wcal$. But since $A$ has non non-trivial idempotents, it follows from \cite[Theorem 2.44]{Lam} that the trace coincides with $A$, so $A_\Wcal=0$.

The maps $\beta$ and $\gamma$ can also be described as follows.  Given a wide subcategory $\Wcal$, we consider the Serre subcategory $\mathcal U= Tr(\mathcal{W})$, where  $Tr$ denotes the Auslander-Bridger-transpose (which is an equivalence on bound modules and thus maps Serre subcategories to Serre subcategories). Then $\gamma(\Wcal)$ is the localisation $A\to Q_{{\mathcal L}_{\mathcal{U}}}$ associated with the Gabriel topology $\mathcal{L}_\Ucal$, and the tilting class $\beta(\Wcal)$ consists of the ${\mathcal L}_{\mathcal{U}}$-divisible modules (as defined in \cite[p.155]{Ste}).

Finally, we observe that the poset of homological ring epimorphisms of a Dedekind domain $A$ is dual to the poset of subsets of maximal ideals of $A$ (see \cite[Corollary 6.12]{AS1}).
\end{example}
 
We finish the paper by placing our classification of minimal tilting modules over hereditary rings in the context of further classification results for some special classes of rings. 

\begin{remark}\label{other}
If $A$ is a Pr\"ufer domain, the maps $\beta$ and $\gamma$ above define bijections between:
\begin{enumerate}
\item[(a)] equivalence classes of tilting $A$-modules;
\item[(b)] wide subcategories of bound $A$-modules;
\item[(c)] perfect Gabriel topologies on $A$;
\item[(d)] epiclasses of non-trivial universal localisations of $A$.
\end{enumerate}
Moreover, all non-trivial universal localisations are injective homological ring epimorphisms, but the converse is not true in general (see \cite[Section 8]{BS}).

\smallskip

Indeed, the assignment $\gamma$ is a bijection between (b) and (d): one can check that the arguments in \cite{Scho2} also hold for semihereditary rings. The correspondence between (c) and (d) follows as in Example \ref{Ded}. In particular, we can {deduce} that wide and Serre subcategories of bound modules coincide. Finally, the bijection between (a) and (c) is \cite[Theorem 5.3]{BET}.
\end{remark}

\begin{remark}
If $A$ is a commutative noetherian ring, there are bijections between:
\begin{enumerate}
\item[(a)] equivalence classes of tilting $A$-modules;
\item[(b)] wide subcategories of bound $A$-modules;
\item[(c)] faithful Gabriel topologies of $A$.
\end{enumerate}
A Gabriel topology is called \textbf{faithful} if the  localisation  $A\to Q_{\mathcal L}$ is injective. The bijection between (a) and (c) is  \cite[Theorem 2.11]{APST}. Again, to a wide subcategory of bound modules $\Ucal$, we associate the Gabriel topology ${\mathcal L}_{\mathcal U}$ and the tilting class consisting of the ${\mathcal L}_{\mathcal{U}}$-divisible modules.
\end{remark}

Our results for hereditary rings thus  share some common features with further classifications over other classes of rings. It would be nice to have a general statement encompassing all these cases.


\begin{thebibliography}{99}
\bibitem[AIR]{AIR} {\sc T.~Adachi, O.~Iyama, I.~Reiten}, \emph{$\tau$-tilting theory}, Compos. Math. {\bf 150} (2014), 415--452.
\bibitem[AA]{AA} {\sc L.~Angeleri H\"ugel, M.~Archetti}, {\em Tilting modules and universal localisation}, Forum Math. {\bf 24} (2012) 709-731.
\bibitem[AHT]{AHT}{\sc L.~Angeleri H\"ugel, D.~Herbera,  J.~Trlifaj}, \emph{Tilting modules and Gorenstein rings.}  {Forum Math.}, {\bf 18} (2006), 217--235.
\bibitem[AMV]{AMV} {\sc L.~Angeleri H\"ugel, F.~Marks, J.~Vit\'oria}, {\em Silting modules}, preprint, arXiv:1405.2531.
\bibitem[APST]{APST} {\sc L.~Angeleri H\"ugel, D.~Posp{\'\i}{\v s}il, J.~\v{S}\v{t}ov\'\i\v{c}ek,  J.~Trlifaj}, \emph{Tilting, cotilting, and spectra of commutative noetherian rings}, Trans. Amer. Math. Soc. {\bf 366} (2014), 3487--3517.
\bibitem[AS1]{AS1} {\sc L.~Angeleri H\"ugel, J.~S\'anchez}, \emph{Tilting modules arising from ring epimorphisms}, Algebr. Represent. Theory \textbf{14} (2011), no. 2, 217--246.
\bibitem[AS2]{AS2}{\sc L.~Angeleri H\"ugel,  J.~S\'anchez}, \emph{Tilting modules over tame hereditary algebras}, {Journal Reine Angew. Math.} {\bf 682} (2013), 1--48.
\bibitem[ATT]{ATT} {\sc L.~Angeleri H\"ugel, A.~Tonolo, J.~Trlifaj}, \emph{Tilting preenvelopes and cotilting precovers}, Algebr. Represent. Theory \textbf{4} (2001), 155-170.
\bibitem[AT]{AT0} {\sc L.~Angeleri H\"ugel, J.~Trlifaj}, \emph{Tilting theory and the finitistic dimension conjectures}, Trans. Amer. Math. Soc. \textbf{354} (2002), 4345--4358.
\bibitem[B]{B} {\sc S.Bazzoni}, \emph{Equivalences induced by infinitely generated tilting modules}, Proc. Amer. Math. Soc. \textbf{138} (2010), 533-544.
\bibitem[BET]{BET} {\sc S.~Bazzoni, P.C.~Eklof, J.~Trlifaj}, \emph{Tilting cotorsion pairs}, Bull. London Math. Soc. \textbf{37} (2005), 683--696.
\bibitem[BS]{BS} {\sc S.~Bazzoni, J.~\v{S}\v{t}ov\'\i\v{c}ek}, \emph{Smashing localisations of rings of weak global dimension at most one}, preprint, arXiv:1402.7294.
\bibitem[C]{Cohn}{\sc P. M. Cohn}, {\it Free Rings and Their Relations}, second ed., London Math. Soc. Monographs 19, Academic Press Inc. London, 1985.
\bibitem[CTT1]{CTT} {\sc R.~Colpi, A.~Tonolo, J.~Trlifaj}, {\em Partial cotilting modules and the lattices induced by them}, Comm. Algebra {\bf 25} (1997), no. 10, 3225--3237. 
\bibitem[CTT2]{CTT2} {\sc R.~Colpi, A.~Tonolo, J.~Trlifaj}, {\em Perpendicular categories of infinite dimensional partial tilting modules and transfers of tilting torsion classes}, J. Pure Appl. Algebra {\bf 211} (2007), no. 1, 223--234.
\bibitem[CT]{CT} {\sc R. Colpi, J. Trlifaj}, \emph{Tilting modules and tilting torsion theories}, J. Algebra \textbf{178} (1995), 614--634.
\bibitem[CB]{CB} {\sc W. Crawley-Boevey}, \emph{Regular modules for tame hereditary algebras.} {Proc. London Math. Soc.}, \textbf{62} (1991), 490--508.
\bibitem[GdP]{GdP} {\sc P.~Gabriel, J.~de la Pe\~na}, {\em Quotients of representation-finite algebras}, Comm. Algebra {\bf 15} (1987), no. 1-2, 279--307.
\bibitem[GL]{GeLe} \textsc{W.~Geigle, H.~Lenzing}, {\em  Perpendicular categories with applications to representations and sheaves}, J. Algebra {\bf 144}, (1991), no.2, 273--343.
\bibitem[GT]{GT} {\sc R.~Gentle, G.~Todorov}, {\em Extensions, kernels and cokernels of homologically finite subcategories} in: Representation theory of algebras (Cocoyoc, 1994), 227--235, CMS Conf. Proc., {\bf 18}, Amer. Math. Soc., Providence, RI, 1996. 
\bibitem[H]{He} {\sc I. Herzog}, \emph{The Ziegler spectrum of a locally coherent Grothendieck category}, {Proc. London Math. Soc.} \textbf{74} (1997), 503--558.
\bibitem[HK]{HK} {\sc A. Hubery, H. Krause}, \emph{A categorification of non-crossing partitions}, preprint, arXiv:1310.1907.
\bibitem[IT]{IT} {\sc C.~Ingalls, H.~Thomas}, {\em Noncrossing partitions and representations of quivers}, Compos. Math. {\bf 145} (2009), no. 6, 1533--1562.
\bibitem[J]{J} {\sc G.~Jasso}, \emph{Reduction of $\tau$-tilting modules and torsion pairs}, to appear in Int. Math. Res. Not. IMRN (2014), doi: 10.1093/imrn/rnu163.
\bibitem[K]{K}{\sc H. Krause}, \emph{The spectrum of a locally coherent category}, {J. Pure Appl. Algebra} \textbf{114} (1997), 259--271.
\bibitem[KS]{KS}{\sc H.~Krause, M.~Saor\'{\i}n}, \emph{On minimal approximations of modules}, In: Trends in the representation theory of finite dimensional algebras (ed. by E.~L.~Green and B.~Huisgen-Zimmermann), Contemp. Math. {\bf 229} (1998) 227-236.
\bibitem[KSt]{KSt} {\sc H.~Krause, J.~{\v S}{\v{t}}\!ov{\'{\i}}{\v{c}}ek}, {\em The telescope conjecture for hereditary rings via Ext-orthogonal pairs}, Adv. Math. {\bf 225} (2010), no. 5, 2341--2364.
\bibitem[Lam]{Lam} {\sc T.Y.~Lam}, \emph{Lectures on Modules and Rings}, Springer (1999), xxiv+557.
\bibitem[L1]{L1} {\sc F.~Lukas}, {\em Infinite-dimensional modules over wild hereditary algebras}, J. London Math. Soc. {\bf 44} (1991), 401--419.
\bibitem [L2]{L2} {\sc F.~Lukas}, {\em A class of infinite-rank modules over tame hereditary algebras}, J. Algebra {\bf 158} (1993), 18--30.
\bibitem[M]{Ma} {\sc F.~Marks}, {\em Universal localisations and tilting modules for finite dimensional algebras}, J. Pure Appl. Algebra {\bf 219} (2015), 3053-3088.
\bibitem[MS]{MS} {\sc F. Marks, J. \v{S}\v{t}ov\'\i\v{c}ek}, \emph{Torsion classes, wide subcategories and localisations}, preprint, arXiv:1503.04639.
\bibitem[R]{R}{\sc C.~M.~Ringel}, {\em The Catalan combinatorics of the hereditary artin algebras}, preprint, arXiv:1502.06553.
\bibitem[Sch1]{Sch} {\sc A.~H.~Schofield}, {\em Representations of rings over skew fields}, Cambridge University Press (1985).
\bibitem[Sch2]{Sch2} {\sc A.~H.~Schofield}, {\em Severe right Ore sets and universal localisation}, preprint, arXiv:0708.0246.
\bibitem[Sch3]{Scho2} {\sc A.~H.~Schofield}, {\em Universal localisations of hereditary rings}, preprint, arXiv:0708.0257.
\bibitem[Ste]{Ste} {\sc B.~Stenstr\"om}, {\em Rings of Quotients}, Springer-Verlag (1975), viii+309pp.
\bibitem[T]{Tak}{\sc R. Takahashi}, \emph{Classifying subcategories of modules over a commutative Noetherian ring}, J. Lond. Math. Soc. (2) {\bf 78} (2008), 767--782.
\bibitem[X]{Xu} {\sc J.~Xu}, {\em Flat covers of modules}, Lecture Notes in Mathematics {\bf 1634}, Springer-Verlag (1996), x+161 pp.
\end{thebibliography}
\end{document}